\documentclass[12pt]{amsart}
\usepackage{amsxtra}
\usepackage{amscd}
\usepackage{amsthm}
\usepackage{amsfonts}
\usepackage{amssymb}
\usepackage{eucal}

\usepackage{xcolor}

\input prepictex
\input pictex
\input postpictex

\usepackage{mathptm}

\usepackage{verbatim}

\usepackage{color}

\newtheorem{claim}{\indent Claim}
\newtheorem{theorem}{Theorem}[section]
\newtheorem*{theorem*}{Theorem}
\newtheorem{lemma}[theorem]{Lemma}
\newtheorem{proposition}[theorem]{Proposition}
\newtheorem*{proposition*}{Proposition}
\newtheorem{corollary}[theorem]{Corollary}
\newtheorem{definition}[theorem]{Definition}
\newtheorem{remark}[theorem]{Remark}

\numberwithin{equation}{section}

\newcommand{\Z}{{\mathbb Z}}

\newcommand{\C}{{\mathbb C}}
\newcommand{\K}{{\C}}

\newcommand{\fb}{{\mathfrak b}}
\newcommand{\fg}{{\mathfrak g}}
\newcommand{\fh}{{\mathfrak h}}
\newcommand{\fn}{{\mathfrak n}}

\newcommand{\U}{{\rm U}}
\newcommand{\Uq}{\widehat{{\rm U}}_q}

\def\a{\alpha}

\def\d{\delta}
\def\D{\Delta}

\def\l{\lambda}

\def\o{\omega}

\def\es{\epsilon}

\begin{document}

\title[Integrable representations of quantum affine superalgebras]
{Integrable representations of the \\ quantum affine special linear superalgebra}

\author{Yuezhu Wu}
\author{R. B. Zhang}
\address[Wu]{School of Mathematics and Statistics, Changshu Institute of Technology, Changshu, Jiangsu, China}
\address[Wu, Zhang]{School of Mathematics and Statistics, University of Sydney, Sydney, NSW 2006, Australia}
\email{yuezhuwu@maths.usyd.edu.au}
\email{ruibin.zhang@sydney.edu.au}

\begin{abstract}
The simple integrable modules with finite dimensional weight spaces are classified
for the quantum affine special linear superalgebra $\U_q(\widehat{\mathfrak{sl}}(M|N))$
at generic $q$.
Any such module is shown to be a highest weight or lowest weight module with respect
to one of the two natural triangular decompositions of the quantum affine superalgebra
depending on whether the level of the module is zero or not.
Furthermore, integrable $\U_q(\widehat{\mathfrak{sl}}(M|N))$-modules
at nonzero levels exist only if $M$ or $N$ is $1$.

\noindent{\bf Key words:} quantum supergroups, quantum affine superalgebras,
integrable modules, highest weight modules.
\end{abstract}
\maketitle


\section{Introduction}\label{sect:intro}

Quantum supergroups associated with simple Lie superalgebras and
their affine analogues were introduced \cite{BGZ, y94, ZGB91b}
(see also \cite{CK, FLV, S})  in the early 90s,  and their structure and representations
have since been  extensively developed (see, e.g., \cite{BKK,
KT, La, MZ, PSV, WZ, YZ, Z, Z92, Z93, Z97, Z98, Zo98, Zo01}). 
Quantum supergroups were applied to
solve interesting problems in a variety of areas such as
topology of knots and $3$-manifolds \cite{LGZ, Z92a, Z95},
quantum supergeometry \cite{Z98, Z04}, and in particular,
Yang-Baxter type integrable models \cite{BGZ, FK, ZBG91, YZ}, where
the problem of constructing solutions of the spectral parameter dependent
Yang-Baxter equation was converted 
to the much easier linear problem of solving the
$\Z_2$-graded Jimbo equations \cite{BGZ} by 
using the representation theory of quantum supergroups. 

The $\Z_2$-Jimbo equations determine the universal $R$-matrix \cite{KT} 
of quantum affine superalgebras in loop representations.
A basic problem in studying the equations is to determine which finite dimensional
irreducible representation of a quantum supergroup can be lifted to a
representation of the corresponding quantum affine superalgebra.
It was shown that the natural representations of quantum orthosymplectic supergroups
can be lifted \cite{ZBG91}, and more importantly,
every finite dimensional irreducible representation of the quantum general
linear supergroup $\U_q(\mathfrak{gl}(M|N))$ \cite{Z93}
can be lifted to an irreducible representation of the untwisted quantum affine general
linear superalgebra $\U_q(\widehat{\mathfrak{gl}}(M|N))$ \cite{Z92}.

In a very recent paper \cite{zh13},  Huafeng Zhang gave a classification of
the finite dimensional simple modules for
$\U_q(\widehat{\mathfrak{sl}}(M|N))$ (more precisely the subalgebra
$\U'_q(\widehat{\mathfrak{sl}}(M|N))$ without the degree operator) at generic $q$,
providing a parametrisation of such simple modules in terms of  highest weight polynomials.
This has much similarity to the classification \cite{Z96}
of  finite dimensional simple modules  for the ${\mathfrak{gl}}(M|N)$ super Yangian,
as explained in \cite{zh13}.

The present paper  generalises results of \cite{r11, rz04, wz14}
on  $\widehat{\mathfrak{sl}}(M|N)$ to the quantum setting to obtain a classification of
the simple integrable modules with finite dimensional
weight spaces for $\U_q(\widehat{\mathfrak{sl}}(M|N))$ at generic $q$.
A module for a quantum affine superalgebra $\U_q(\widehat{\fg})$ is integrable
if it is integrable with respect to the subalgebra $\U_q(\widehat{\fg}_{\bar0})$,
which is the quantised universal enveloping algebra of
the even subalgebra $\widehat{\fg}_{\bar0}$ of $\widehat{\fg}$.
Thus the integrability of a $\U_q(\widehat{\mathfrak{sl}}(M|N))$-module
amounts to integrability with respect to the subalgebras
$\U_q(\widehat{\mathfrak{sl}}(M))$ and $\U_{q^{-1}}(\widehat{\mathfrak{sl}}(N))$.
The requirement of having finite dimensional weight spaces imposes further stringent
conditions on the module.

One result of this paper,  Theorem \ref{thm:hw},  states  that a zero-level simple integrable module with finite dimensional
weight spaces is necessarily of highest weight type
with respect to the triangular decomposition of $\U_q(\widehat{\mathfrak{sl}}(M|N))$
induced by the distinguished triangular decomposition of $\mathfrak{sl}(M|N)$ (cf. equation \eqref{loop-tri-decomp}).
A classification of such modules is given in terms of their highest weight polynomials (see Theorem \ref{thm:class}).

We show in Theorem \ref{thm:class} that any simple integrable $\U_q(\widehat{\mathfrak{sl}}(M|N))$-module
$V$ of zero level with finite dimensional weight spaces can be embedded in a quantum loop module (cf \eqref{L-module})
as a direct summand. By setting the loop parameter to $1$, we obtain from the image of $V$ a
 finite dimensional evaluation $\U'_q(\widehat{\mathfrak{sl}}(M|N))$-module
(cf. \eqref{evaluat}). This way we recover all the finite dimensional
simple $\U'_q(\widehat{\mathfrak{sl}}(M|N))$-modules, which were classified in \cite{zh13}.

We prove in Theorem \ref{thm:level}  that
only when $M$ or $N$ is equal to $1$, $\U_q(\widehat{\mathfrak{sl}}(M|N))$ admits
integrable modules with finite dimensional
weight spaces at nonzero levels.  Such a simple integrable module is necessarily a
highest or lowest weight module
with respect to the standard triangular decomposition of $\U_q(\widehat{\mathfrak{sl}}(M|N))$
given in Proposition \ref{tri-decomp}.
The necessary and sufficient condition for a simple highest weight $\U_q(\widehat{\mathfrak{sl}}(M|N))$-module
to be integrable with finite dimensional weight spaces is that the highest weight is
integral and dominant \cite{cp91, L93} with respect to the quantised universal enveloping algebra
$\U_q(\widehat{\mathfrak{sl}}(M|N)_{\bar 0})$ of the even subalgebra of $\widehat{\mathfrak{sl}}(M|N)$.

We mention that the quantised universal enveloping superalgebras of symmetrizable
affine Lie superalgebras (without isotropic odd simple roots)
admit many more integrable highest weight modules
at nonzero levels.  A classification of such simple modules was obtained in \cite{Z97},
where  a ``super duality" was discovered identifying such
quantised universal enveloping superalgebras with certain classes of ordinary
quantum affine algebras.

\section{Preliminaries}\label{sect:prelim}

In order to study the integrable modules for the quantum affine special linear superalgebra
$\U_q(\widehat{\mathfrak{sl}}(M|N))$, we need its loop presentation \cite{y99},
which we discuss here.

\subsection{} Let us start by discussing some basic structural properties of the special linear superalgebra
\cite{k77}. Fix positive integers $M$ and $N$, and assume that at least one of them is greater than $1$.
Let $ I$ be the set $\{1,2,\dots, M+N-1\}$.
We choose the distinguished Borel subalgebra $\fb$ for $\mathfrak{sl}(M|N)$,
which consists of the upper triangular matrices.
The Cartan subalgebra $\fh\subset\fb$ consists of the diagonal matrices in $\mathfrak{sl}(M|N)$.
Let $\fn$ be the strictly upper triangular matrices, then $\fb=\fh\oplus\fn$.

Equip the free $\Z$-module $\oplus_{i=1}^{M+N}\Z \es_i$
with the following bilinear from
\[
(\es_i,\es_j) =l_i\d_{ij} ,\quad  l_i
=\left\{\begin{array}{ll}
    1, & \mbox{ if } 1 \le i \le M,\\
    -1,  & \mbox{ if }  M + 1 \le  i \le  M + N.
    \end{array}    \right.
\]
Then the roots of $\mathfrak{sl}(M|N)$ can be expressed as $\es_i-\es_j$ for all $i\ne j$,
and the simple roots with respect to $\fb$ are given by
$\{\a_i:=\es_i-\es_{i+1}|i\in I\}.$
The even subalgebra $\mathfrak{sl}(M|N)_{\bar 0}$ of $\mathfrak{sl}(M|N)$ is $\mathfrak{sl}(M)\oplus\C z\oplus\mathfrak{sl}(N)$,
where $\C z$ is the center of $\mathfrak{sl}(M|N)_{\bar 0}$.
Let $\mathfrak{h}_1$ (resp. $\mathfrak{h}_2$)
be the  Cartan subalgebra of $\mathfrak{sl}(M)$ (resp.
$\mathfrak{sl}(N)$), and denote by $\D_0^1$ (resp. $\D_0^2$) the corresponding set of roots.
Denote by $Q$ the root lattice of  $\mathfrak{sl}(M|N)$, and set
$Q^+=\sum_{i\in I}\Z_{\ge 0}\a_i$.

Let $\widehat{\mathfrak{sl}}(M|N)=\mathfrak{sl}(M|N)\otimes \C[t,t^{-1}]\oplus \C c\oplus\C d$
be the untwisted extended affine Lie superalgebra associated with $\mathfrak{sl}(M|N)$, where $c$ spans
the center, and $d$ is the degree operator.  We take the
following Cartan subalgebra $\widehat{\fh}=\fh\otimes 1\oplus \C c\oplus \C d$ for $\widehat{\mathfrak{sl}}(M|N)$.
Introduce the affine weight $\o_0$ and null root $\d$ in $\fh^*$ such that
$
\o_0(c)=1, \ \delta(d)=1$,  and $\o_0(h)=\delta(h)=0,  \  \forall h\in \fh.
$
Then
\[
(\o_0,\a_i)=(\d,\a_i)=0, \  \forall i\in I, \quad (\o_0,\o_0)=(\d,\d)=0, \quad (\o_0,\d)=1.
\]
Then $\o_0$, $\d$ and all the $\a_i$ together form a basis of $\fh^*$.
%
Denote by $\widehat{Q}$ the $\Z$-span of the $\alpha_i$ and $\delta$, i.e.,
the root lattice of  $\widehat{\mathfrak{sl}}(M|N)$,  and let $\a_0=\d-\sum_{i\in I}\a_i $.

Recall that we have the following Borel subalgebras of $\widehat{\mathfrak{sl}}(M|N)$,
\begin{eqnarray}
&&\C c\oplus \C d\oplus\fb\oplus {\mathfrak{sl}}(M|N)\otimes t\C[t], \label{eq:borel-1}\\
&&\C c\oplus \C d\oplus\fb\otimes\C[t, t^{-1}], \label{eq:borel-2}
\end{eqnarray}
where \eqref{eq:borel-1} is the standard
Borel subalgebra, while \eqref{eq:borel-2} is induced by $\fb\subset\mathfrak{sl}(M|N)$.
Later we will make use of quantum universal enveloping superalgebras of these Borel subalgebras.

\subsection{}
Let us fix once for all a nonzero complex number $q$ which is not a root of $1$.
For any $m\in\Z_+$, define
$[m]_q=\frac{q^m-q^{-m}}{q-q^{-1}}$.
Let $q_i = q^{(\es_i,\es_i)}$ for $i\in I$, and set
\[
a_{ij}= (\es_i - \es_{i+1}, \es_j -\es_{j+1}), \quad \mbox{for all  }  i, j \in I.
\]

The quantum affine superalgebra  $\Uq:=\U_q(\widehat{\mathfrak{sl}}(M|N))$ is
a Hopf superalgebra over $\K$ \cite{BGZ, KT, y99, Z92, ZGB91b}, which has two presentations,
a Serre presentation in terms of Chevalley generators and Serre type relations,
and loop presentation.
Its loop presentation was constructed in \cite{y99}.
\begin{definition}\label{def:qaff}
The loop presentation of  $\Uq:=\U_q(\widehat{\mathfrak{sl}}(M|N))$ is as follows.
The set of generators is
 \[
\{X^{\pm }_i(n),K^{\pm 1}_i,  h_i(s), C^{\pm1/2}, D^{\pm 1}\ |\ i\in I, n, s  \in\Z, s\ne 0\},
\]
where $X^{\pm }_M(m)$ for all $m\in\Z$ are odd, and the other elements are even.

The relations are
\begin{eqnarray*}
&&C^{\pm 1/2} \mbox{ are central},\\
&&K_iK^{-1}_i = 1 = K^{-1}_i K_i,\quad  [K_i,K_j ] = [K_i, h_j(s)] = 0,\quad [K_i,D]=0,\\
&&DD^{-1}=D^{-1}D=1, \quad
Dh_i(s)D^{-1}=q^sh_i(s),\quad DX^{\pm}_i(s)D^{-1}=q^sX^{\pm}_i(s),\\
&& K_iX^{\pm}_j(n)K_i^{-1} = q^{\pm a_{ij}}X^{\pm}_j(n),\\
&&[h_i(m), h_j(n)]=\d_{m+n,0}\frac{[ml_ia_{ij}]_{q_i}(C^m-C^{-m})}{m(q_i-q_i^{-1})},\\
&&[h_i(s),X^{\pm}_j(n)] = \pm\frac{[sl_ia_{ij}]_{q_i}}{s}C^{\mp\frac{|s|}{2}}X^{\pm}_j(n+s),\\
&&[X^+_i(m), X^-_j(n)] =\d_{ij}\frac{C^{(m-n)/2}\phi^+_i(m+n) -C^{-(m-n)/2}\phi^-_i(m+n)}{q_i - q^{-1}_i},\\
&&[X^{\pm}_i(m),X^{\pm}_j(n)] = 0 \quad \mbox{  for } a_{ij} = 0, \\
&&[X^{\pm}_i(m+1), X^\pm_j(n)]_{q^{\pm a_{ij}}}+[X^{\pm}_j(n+1)
, X^{\pm}_i(m)]_{q^{\pm a_{ij}}}=0\quad \mbox{  for }a_{ij}\ne 0,
\end{eqnarray*}
and
\begin{eqnarray*}
&& {\rm Sym}_{m,n}[X^{\pm}_i(m), [X^\pm_i(n),X^\pm_j(k)]_{q^{-1} }]_q   = 0
\quad \mbox{  for } a_{ij} = \pm 1, i \ne  M,\\
&&{\rm Sym}_{n,u}[[[X^\pm_{M-1}(m),X^\pm_M(n)]_{q^{-1}} , X^\pm_{M+1}(k)]_q,X^\pm_M(u)] = 0,
\quad \mbox{ when }M,N> 1,
\end{eqnarray*}
where $\phi^\pm_i(n)$ are given by the generating series
\[
\sum\limits_{n\in\Z}\phi^{\pm}_i(n)z^n= K^{\pm 1}_i{\rm exp}(±(q_i - q_i^{-1})
\sum\limits_{s\in \Z_{>0}}h_i(\pm s)z^{\pm s})\in \Uq[[z, z^{-1}]],
\]
and the symbol ${\rm Sym}_{k,l}$ means symmetrization with respect to $k$ and $l$.
We have used the notation of $q$-brackets $[X,Y]_u = XY- (-1)^{|X| |Y|}uY X$,
and written $[X,Y]$ for $[X,Y]_1$ for simplicity.
\end{definition}

We denote by $\U'_q(\widehat{\mathfrak{sl}}(M|N))$ the subalgebra of $\Uq$
without the generators $D^{\pm 1}$.
By dropping the generators $X_M^{\pm}(n)$ for all $n\in\Z$, we obtain a
subalgebra of $\Uq$, which is the quantised universal enveloping algebra
$\U_q(\widehat{\mathfrak{sl}}(M|N)_{\bar0})$ of the even subalgebra
${\mathfrak{sl}}(M|N)_{\bar0}$ of ${\mathfrak{sl}}(M|N)$. Note that this
subalgebra contains $\U_q(\widehat{\mathfrak{sl}}(M))$
and $\U_{q^{-1}}(\widehat{\mathfrak{sl}}(N))$ as subalgebras.

The superalgebra $\Uq$ is $\Z$-graded  $\Uq=\oplus_{k\in\Z}(\Uq)_k$ with
homogeneous components $(\Uq)_k=\{x\in \Uq\ |\ DxD^{-1}=q^kx\}$.
Let us introduce the following subalgebras of $\Uq$:
\begin{itemize}
\item $\Uq^{+}(\gg)$ (resp. $\Uq^{+}(\ll)$) denotes
the subalgebra generated by the elements $X^+_i(m)$ for all $m\ge 0$ and $i\in I$
(resp. $X^+_i(m)$ for all $m<0$ and $i\in I$);
\item $\Uq^{-}(\gg)$ (resp. $\Uq^{-}(\ll)$) denotes
the subalgebra generated by the elements $X^-_i(m)$ for all $ m> 0$ and $i\in I$
(resp. $X^-_i(m)$ for all $ m\le 0$ and $i\in I$);
\item $\Uq^0(\gg)$ (resp. $\Uq^0(\ll)$) denotes
the subalgebra generated by the elements $h_i(r)$ for all $r> 0$ and $i\in I$
(resp. $h_i(r)$ for all $r<0$ and $i\in I$);
\item $\Uq^0$ denotes
the subalgebra generated by $K^{\pm 1}_i$ ($ i\in I$), $D^{\pm 1}$ and $C^{\pm1/2}$.
\end{itemize}
We have the following obvious result.
\begin{proposition}\label{tri-decomp}
Define the following subspaces  of $\Uq$
\[
\Uq(+):=\Uq^+(\gg) \Uq^-(\gg)\Uq^0(\gg), \quad \Uq(-):=\Uq^+(\ll) \Uq^-(\ll)\Uq^0(\ll).
\]
Then $\widehat{\overline{B}}=\Uq(-)\Uq(0)$ and $\widehat{B}=\Uq(0)\Uq(+)$ are subalgebras,  and
\[
\Uq=\Uq(-)\Uq^{0}\Uq(+).
\]
\end{proposition}

Note that $\widehat{B}$ is the quantised universal enveloping algebra of
the Borel subalgebra of $\widehat{\mathfrak{sl}}(M|N)$ given in \eqref{eq:borel-1}.
Thus this triangular decomposition of $\Uq$ is the quantum analogue of the
triangular decomposition of ${\rm U}(\widehat{\mathfrak{sl}}(M|N))$
with respect to the Borel subalgebra \eqref{eq:borel-1}.

\subsection{}
Let $\U_q$ be the extended quantum loop superalgebra, namely the quotient of $\Uq$ by the ideal
generated by $C^{\pm 1/2}-1$, and denote by $\U'_q$ the $\K$-subalgebra of $\U_q$
without the generators $D^{\pm 1}$.  Define the following subalgebras of $\U_q$:
\begin{itemize}
\item $\U_q(0)$ denotes the subalgebra generated by
$h_i(r)$, $K^{\pm 1}_i$, $D^{\pm 1}$ for all $i\in I$ and $0\ne r\in\Z$;
\item  $\U_q^+$
 (resp. $\U_q^-$)  denotes the subalgebra generated by
$X^+_i(m)$ for all $i\in I$ and $m\in \Z$ (resp. $X^-_i(m)$ for all $i\in I$ and $m\in \Z$),
\end{itemize}
and let $\U'_q(0)= \U_q(0)\cap \U'_q$, which is a subalgebra of $\U'_q$. Then
\begin{equation}\label{loop-tri-decomp}
\U_q=\U_q^-\U_q(0)\U^+_q, \quad \U'_q=\U_q^-\U'_q(0)\U^+_q.
\end{equation}
Define the following subalgebras of $\U_q$ and $\U'_q$ respectively:
\begin{equation}\label{eq:B-B'}
B:=\U_q(0)\U_q^+, \quad B':=\U'_q(0)\U_q^+.
\end{equation}
Then $B$ can be considered as the quantised universal enveloping superalgebra of the Borel subalgebra
of $\widehat{\mathfrak{sl}}(M|N)$ given in \eqref{eq:borel-2} without the central element,
and $B'$ is the subalgebra of $B$ without the generators $D^{\pm 1}$.
Thus the triangular decompositions \eqref{loop-tri-decomp} are quantum analogues of the
triangular decomposition of ${\rm U}(\widehat{\mathfrak{sl}}(M|N))$
with respect to the Borel subalgebra \eqref{eq:borel-2} of $\widehat{\mathfrak{sl}}(M|N)$.

Set $\D=\{\beta_{i j}:=\a_i+\a_{i+1}+\cdots +\a_j | i, j\in I, i\le j\}$ with the following total ordering
$\beta_{i, j} <\beta_{i', j'}$ if $i<i'$ or $i=i', j<j'$.
For $\beta_{i, j} \in \D$ and $n\in\Z$, define
\[
X^+_{i,j}(n)=[\cdots[[X^+_i(n), X^+_{i+1}(0)]_{q_{i+1}}, X^+_{i+2}(0)]_{q_{i+2}}, \cdots, X^+_j(0)]_{q_j},
\]
with the convention that $X^+_{i,i}(n)=X^+_i(n).$

\begin{proposition}\cite[Theorem 3.12]{zh13}\label{prop:PBW}
$\U_q^+$ is spanned by vectors of the form
\[
\prod\limits_{1\le a\le b\le M+N-1}^{\rightarrow}\left(\prod\limits_{i=1}^{c_{ab}}X^+_{a,b}(n_{ab,i})\right), \quad
c_{ab}\in\Z_{\ge 1}, \, n_{ab,i}\in\Z,
\]
where $\prod\limits^{\rightarrow}$ is the ordered product positioning $X^+_{a,b}(m)$ on the left of $X^+_{a',b'}(n)$
if $\beta_{a, b}<\beta_{a', b'}$.
\end{proposition}

\subsection{}\label{sect:int-module} All the modules for $\Uq$ and $\Uq'$
considered in this paper are assumed to be $\Z_2$-graded.
Given a $\Uq$-module $V$, let
\[
V_\mu=\{v\in V | D v=q^{(\mu, \omega_0)} v,  \
C^{\pm 1/2} v= q^{\pm\frac{1}{2}(\mu, \delta)} v,  \ K_i v=q_i^{(\mu, \alpha_i)}v, \  i\in I\}
\]
for any $\mu\in \widehat{\fh}^*$.
If $V_\mu\ne 0$, we say that $\mu$ is a weight of $V$, and denote by $P(V)$ the set of the weights.
The module $V$ is said to be {\it  a weight  module of type $1$} if
\[
V=\bigoplus_{\mu\in P(V)}V_\mu.
\]
From now on, all modules will be assumed to be of type $1$. A $\Uq$ module $V$ is {\it integrable}
if $V=\bigoplus_{\mu\in P(V)}V_\mu$, and  the elements $X^{\pm}_i(m)$  ($i\in I$, $m\in\Z$)
act locally nilpotently.  If $C^{\pm 1/2}$ act by the identity, we say that $V$ is a {\em zero-level} module,
or {\em at level $0$}.

\section{Zero-level integrable representations for $\Uq$}

In this section we classify  the irreducible integrable $\Uq$-modules with finite dimensional
weight spaces such that $C^{\pm 1/2}$ act as the identity. Such modules descend to $\U_q$-modules.

\subsection{}
Let $H$ (resp. $H'$) be the subalgebra of $\U_q$ generated by $K^{\pm 1}_i,  D^{\pm 1}, i\in I$
(resp. $K^{\pm 1}_i,  i\in I$).
A module $V$ of $\U_q$ (resp. $\U'_q$) is called a {\it highest
weight module} if there exists a nonzero weight vector $v\in V$ with respect
to $H$ (resp. $H'$) such that
\begin{enumerate}
\item
$\U_q v=V$ (resp. $\U'_qv=V$),
\item $X_i^+(m)v=0$ for all $i\in I$ and $m\in \Z$,  and
\item
$\U_q(0)v$ (resp.  $\U'_q(0)v$) is an irreducible
$\U_q(0)$-module (resp.  $\U'_q(0)$-module).
\end{enumerate}
Call $v$ a highest weight vector of $V$ relative to $B$ (resp. $B'$) as
$B v=\K v$ (resp. $B' v=\K v$).
These highest weight modules are defined relative to the triangular decompositions
for $\U_q$ (resp. $\U'_q$) defined by \eqref{loop-tri-decomp}.

Let $\psi: \U'_q(0)\rightarrow \K$ be any algebra homomorphism,
and let $\U'_q(0)$ act on the one dimensional vector
space $\K_\psi=\K$ by $\psi$.  We extend $\K_\psi$ to a module over $B'$ (cf. \eqref{eq:B-B'})
by letting $\U_q^+$ act trivially.
Construct the induced $\U'_q$-module
\[
M(\psi)=\U'_q\otimes_{B'}\K_\psi,
\]
which has a unique simple quotient:
\begin{eqnarray}\label{eq:ev-VPsi}
 V(\psi)=\text{the simple quotient $\U'_q$-module of $M(\psi)$}.
\end{eqnarray}

The following definition is taken from \cite{zh13}.
\begin{definition}\cite{zh13}
Let $\mathcal{R}_{M,N}$ be the set consisting of elements $(P, f, c, Q)$,  where
\begin{enumerate}
\item $f(z)=\sum\limits_{n\in\Z}f_nz^n\in\K[[z,z^{-1}]]$ is a formal series and
$Q(z)\in 1+z\K[z]$ is a polynomial such that
\[
Q(z)f(z)=0;
\]
\item $c\in \C\setminus \{0\}$ with $\frac{c-c^{-1}}{q-q^{-1}}=f_0$;
\item $P=(P_1,\dots, P_{M-1}, P_{M+1},\dots, P_{M+N-1})$ with $P_i\in 1+z\K[z]$.
\end{enumerate}
\end{definition}

With the help of results from \cite{zh13},
we can characterise the integrability of $V(\psi)$ as follows.

\begin{theorem}\label{thm:fd-module}
The following  are equivalent for the simple $\U'_q$-module $V(\psi)$ (cf. \eqref{eq:ev-VPsi}).
\begin{enumerate}
\item $V(\psi)$ is an integrable $\U'_q$-module with finite dimensional weight spaces.
\item
There exists $(P, f, c, Q)\in \mathcal{R}_{M,N}$ such that for any
 highest weight vector $v\in V(\psi)$,
\begin{eqnarray}
&&  X^+_i(n) v=0\quad \mbox{ for } i\in I, n\in\Z,\\
&& \label{Pi}
\psi\left(\sum\limits_{n\in\Z}\phi^{\pm}_i(n)z^n\right) v=q_i^{{\rm deg}P_i}\frac{P_i(zq^{-1}_i)}{P_i(zq_i)}v
\in\K[[z^{\pm 1}]],\quad  i\in I, i\ne M,\\
&&   \label{-}  (X^-_i(0))^{1+{\rm deg}P_i}v=0,\quad  i\in I, i\ne M,\\
&&\label{M-f}  \psi(K_M)v=cv,\quad \psi\left(\sum\limits_{n\in\Z}\frac{\phi^+_M(n)-\phi^-_M(n)}{q-q^{-1}}z^n\right)v=f(z)v,\\
&&\label{M-Q}\sum\limits_{s=0}^da_{d-s}X^-_M(s+r)v=0, \ \  \forall r\in\Z, \  \text{with $Q(z)=\sum\limits_{s=0}^da_sz^s$},
\end{eqnarray}
where \eqref{Pi}  is understood as an equation of Laurent series expanded
about $z =0$ for $\phi^+_i$ (resp.  $z =\infty$ for $\phi^-_i$).
\item  $V(\psi)$ is finite dimensional.
\end{enumerate}
\end{theorem}
\begin{proof}
$(1)\Rightarrow  (2)$.  When $i\ne M$,  let $\U_q^{(i)}$ be the $\U_{q_i}(\widehat{\mathfrak{sl}}(2))$
subalgebra generated by
\[X^{\pm}_i(n),  K^{\pm 1}_i,  h_i(r),  \quad n, r\in\Z,  r\ne 0.
\]
Then $\U_q^{(i)}v$ is an integrable highest weight $\U_q^{(i)}$-module.
By \cite[Theorem 3.4]{cp91}, there exists a polynomial $P_i\in 1+z\K[z]$
satisfying \eqref{Pi} and \eqref{-}.

When $i=M$, there exist $c\in \C\setminus \{0\}, f_n\in\K, n\in \Z$ such that
\begin{eqnarray*}
&&K_Mv=cv,\\
&&\frac{\phi^+_M(n)-\phi^-_M(n)}{q-q^{-1}}v
= \psi\left(\frac{\phi^+_M(n)-\phi^-_M(n)}{q-q^{-1}}\right)v=f_nv,\quad n\ne 0,\\
&&\frac{\phi^+_M(0)-\phi^-_M(0)}{q-q^{-1}}v
= \psi\left(\frac{K_M-K^{-1}_M}{q-q^{-1}}\right)v=f_0v.
\end{eqnarray*}
Since $X^-_M(n)v$ for all $n\in\Z$ belong to the same weight space of $V(\psi)$, and all
weight spaces are finite dimensional,
there exist $m\in\Z, d\in\Z_{\ge 0}$ and $a_0,\dots,a_d\in\K$ satisfying $a_d\ne 0$ and $a_0=1$
such that
\[
\sum\limits_{s=0}^da_{d-s}X^-_M(s+m)v=0.
\]
Applying $h_{M-1}(r)$ to the above equality, we obtain
\begin{eqnarray*}
0&=&\sum\limits_{s=0}^da_{d-s}[h_{M-1}(r), X^-_M(s+m)]v
    +\sum\limits_{s=0}^da_{d-s}X^-_M(s+m)h_{M-1}(r)v\\
&=&\sum\limits_{s=0}^da_{d-s}\frac{[r]_q}{r}X^-_M(s+m+r)v
    + \sum\limits_{s=0}^da_{d-s}X^-_M(s+m)\psi(h_{M-1}(r))v\\
&=&\frac{[r]_q}{r}\left(\sum\limits_{s=0}^da_{d-s}X^-_M(s+m+r)v\right).
\end{eqnarray*}
Hence, $\sum\limits_{s=0}^da_{d-s}X^-_M(s+m+r)v=0$ and  \eqref{M-Q} holds.

Applying $X^+_M(0)$ to $\sum\limits_{s=0}^da_{d-s}X^-_M(s+m+r)v=0$,  we have
\begin{eqnarray*}
X^+_M(0)\sum\limits_{s=0}^da_{d-s}X^-_M(s+m+r)v&=&
\sum\limits_{s=0}^da_{d-s}\frac{\phi^+_M(s+m+r)-\phi^-_M(s+m+r)}{q-q^{-1}}v\\
&=&\sum\limits_{s=0}^da_{d-s}f_{s+m+r}v=0,
\end{eqnarray*}
which implies that  $Q(z)(\sum_{n\in\Z}f_nz^n)=0,$ where $Q(z)=\sum_{s=0}^da_sz^s.$

$(2)\Rightarrow (3)$.  This was established in \cite[Theorem 4.5]{zh13}, which is a key result in the
classification of finite dimensional simple $\U'_q$-modules.

$(3)\Rightarrow (1)$. Clear.
\end{proof}
\begin{definition}\label{def:fd-module}
We will denote  by $V(P, f, c, Q)$ the $\U'_q$-module $V(\psi)$ corresponding to
$(P, f, c, Q)\in \mathcal{R}_{M,N}$ in Theorem \ref{thm:fd-module}, and call $P$,
$f$ and $Q$ the highest weigh polynomials of  $V(P, f, c, Q)$.
\end{definition}

Note that  $f$ is a formal Laurent series in general.

\subsection{ }
The following result is \cite[Lemma 1.4]{cg03}.
\begin{lemma}\label{Cartan}
Let $\chi: \U'_q(0)\rightarrow \K[t,t^{-1}]$ be a nontrivial homomorphism of $\Z$-graded algebras.
Then there exists a unique $r>0$ such that the image of $\chi$ equals to $\K[t^r,t^{-r}].$
\end{lemma}

Let $\tilde\varphi : \U'_q(0)\rightarrow L:=\K[t,t^{-1}]$ be a $\Z$-graded algebra homomorphism
such that $\tilde\varphi(C^{\pm 1/2})=1$ and $\tilde\varphi(K_i^{\pm 1})\in
\K\setminus  \{0\}$. Then for any given $b\in\C$, we can turn $L$ into
a $\U_q(0)$-module via $\tilde\varphi$ defined for all $f\in L$ by
\begin{eqnarray}\label{eq:T-mod}
D f = q^{t\frac{d }{d t} + b}f, \quad x f=\tilde\varphi(x) f, \quad x\in \U'_q(0).
\end{eqnarray}
We write $\varphi=(\tilde\varphi, b)$ and denote by $L_\varphi$ the image of  $\tilde\varphi$
regarded as a $\Z$-graded $\U_q(0)$-submodule. Then $L_\varphi$ is $L_0:=\K$ or a Laurent subring
$L_r:=\K[t^r,t^{-r}]$ for some integer $r> 0$. This follows from Lemma \ref{Cartan}.

Assume that  $L_\varphi$ is a simple $\U_q(0)$-module.  We extend $L_\varphi$ to a module over $B$ (cf. \eqref{eq:B-B'})
with $\U_q^+$ acting trivially, and construct the induced $\U_q$-module
\begin{eqnarray}\label{eq:Verma}
M(\varphi)=\U_q\otimes_{B}L_\varphi.
\end{eqnarray}
This has a unique irreducible quotient, which we denote by $V(\varphi)$.  Then every
irreducible highest weight $\U_q$-module is isomorphic to some $V(\varphi)$.

Given any simple $\U'_q$-module $ V(\psi)$ (cf. \eqref{eq:ev-VPsi}), we
form the vector space $V(\psi)\otimes L$ and denote $w(s)=w\otimes t^s$
for any $w\in V(\psi)$ and $s\in \Z$.  For any $b\in\C$, we now turn $V(\psi)\otimes L$ into a
$\U_q$-module by defining the action
\begin{eqnarray}\label{L-module}
\begin{aligned}
& C^{\pm 1/2} w(s)=w(s), \quad D w(s)=q^{s+b} w(s),\\
& xw(s)=(xw)(s+m), \quad x\in (\U_q)_m.
\end{aligned}
\end{eqnarray}
We denote this $\U_q$-module by $L(V(\psi); b)$
and call it the {\em quantum loop module associated to $V(\psi)$ and $b$}.
Then  $V(\psi)$ is an integrable $\U'_q$-module if and only if  $L(V(\psi);b)$ is an integrable $\U_q$-module.
\begin{theorem}\label{iso}
Let $V(\varphi)$ be a $\Uq$-module such that $L_\varphi\cong L_r$ is
an irreducible $\U'_q(0)$-module.
Define $\psi=S\circ \tilde\varphi: \U'_q(0)\rightarrow \K$
with $S: L \rightarrow \K$, $t\mapsto 1$, being  the evaluation map.
Let  $v$ be a highest weight vector of $V(\psi)$ and denote
$v(i)=v\otimes t^i$ for any $i\in \Z$. Then
\begin{enumerate}
\item
$V(\psi)\otimes L\cong \oplus_{i=0}^{r-1}\U_q v(i)$ as
$\U_q$-modules, where $\U_q$-submodules $\U_q v(i)$ are simple.
Furthermore, $\U_q v(0)\cong V(\varphi)$.
\item $V(\varphi)$ has finite dimensional weight spaces with respect
to $H$ if and only if $V(\psi)$ has finite
dimensional weight spaces with respect to $H'$.
\end{enumerate}
\end{theorem}
\begin{proof}
The proofs of \cite[Theorem 1.8]{r95} and
\cite[Lemma 1.10]{r01} can be adopted verbatim to prove this result.
We refer the interested readers to the paper \cite{r95, r01} for details.
\end{proof}

We note that $\U_q v(i)\cong V(\tilde\varphi, b+i)$. In the case $r=0$, the formula in part (1)
of the theorem should be understood as $V(\psi)\otimes L\cong \oplus_{i\in\Z}\U_q v(i)$.

Given any nonzero simple $\U_q$-submodule $\U_q v(i)\subseteq L(V(\psi), b)$, we define the
{\em evaluation module} for $\U'_q$ by setting $t=1$:
\begin{eqnarray}\label{evaluat}
\U_q v(i) \longrightarrow V(\psi), \quad w(s)\mapsto w.
\end{eqnarray}
This is a $\U'_q$-module homomorphism, which is surjective.

\subsection{ }
Let $V$ be an irreducible integrable $\U_q$-module with finite dimensional weight spaces.
In this section we generalize the method developed in \cite{wz14}
to  show that $V$ has to be  a highest weight module with respect to the
triangular decomposition  of $\U_q$ given in \eqref{loop-tri-decomp}.

Introduce the set $S=\{(a,b)|1\le a\le M\le b \le M+N-1, a<b\}$,
and order the elements so that $(a,b)>(a',b')$ if and only if
$b-a>b'-a'$ or $b-a=b'-a'$, $a<a'.$

We have the following lemmas,
which play a key role in the remainder of the section. Their proofs
are relegated to Appendices \ref{pf:key-1} and \ref{pf:key-2} as
they involve very lengthy computations.
\begin{lemma}\label{key-1}
Let $v_{a-1,b-1}$ be a weight vector in $V$. For any $n_1,\dots, n_p\in\Z$ and $p\in\Z_{>0}$, denote
\begin{eqnarray*}
&& v_{a,b}:=X^+_{a,b}(n_p)\cdots X^+_{a,b}(n_1)v_{a-1,b-1},  \mbox{ \ if \ } b-1\ne M+N,\\
&&({\rm resp. } \  v_{1,b-a}:=X^+_{1,b-a}(n_p)\cdots X^+_{1,b-a}(n_1)v_{a-1,b-1}, \mbox{ \ if \ } b-1= M+N).
\end{eqnarray*}
If
$
X^+_i(m)v_{a-1,b-1}=X^+_{k,l}(m)v_{a-1,b-1}=0$, $\forall i\ne M, m\in\Z, (k,l)>(a-1,b-1),
$
then
\begin{eqnarray*}
&&X^+_i(m)v_{a,b}=X^+_{k,l}(m)v_{a,b}=0 \quad \forall m\in \Z, (k,l)> (a,b), i\ne M,\\
&&({\rm resp. }\    X^+_i(m)v_{1,b-a}=X^+_{k,l}(m)v_{1,b-a}=0 \quad \forall m\in \Z, (k,l)> (1,b-a)).
\end{eqnarray*}
\end{lemma}
\begin{proof} See Appendix \ref{pf:key-1} .
\end{proof}
\begin{lemma}\label{key-2}
For $ (a, b)\in S$, let $v_{a,b}$  be a weight vector in $V$ such that
\begin{eqnarray}\label{2}
\begin{aligned}
&X^+_i(m)(X^+_{a,b}(n_1)\cdots X^+_{a,b}(n_r)v_{a,b})=0, \quad \forall m, n_1,\dots, n_r\in\Z, \ i\ne M,   \\
& X^+_{a,b}(p)X^+_{a,b}(k)v_{a,b}=0, \quad \forall p,k\in\Z \mbox{ with }p\equiv k~({\rm mod}~2).
\end{aligned}
\end{eqnarray}
Then
\begin{equation}\label{3}
X^+_{a,b}(p)X^+_{a,b}(i)X^+_{a,b}(j)v_{a,b}=0, \quad \forall p, i, j\in\Z.
\end{equation}
Furthermore, there exists $k$  $(0\le k\le 2)$ and $n_1,\dots, n_k\in \Z$ such that
\begin{eqnarray}\label{3-}
\begin{aligned}
&w_{a,b}:=  X^+_{a,b}(n_1)\cdots X^+_{a,b}(n_k)v_{a,b}\ne 0,\\
   & X^+_{a,b}(m)w_{a,b}=0,\quad \forall m\in\Z.
\end{aligned}
\end{eqnarray}
\end{lemma}
\begin{proof} See Appendix \ref{pf:key-2} .
\end{proof}

Let $V$ be an irreducible zero-level integrable module for $\U_q$ with finite dimensional weight spaces.
By definition, $V$ is integrable over the even subalgebra of $\U_q$. It follows from
Chari's work \cite{cg03} that there is a non-zero weight vector $v\in V$ such that
\begin{equation}\label{0-hw}
X^+_i(m)v=0, \quad \forall m\in\Z, \ i\ne M.
\end{equation}

Denote by $wt(v)$ the weight of $v$. Let $X$ be the subspace of $V$ spanned by the vectors
$X^+_M(k)X^+_M(-k)v$ for all $k\ge 0$, which is a subspace of $V_{wt(v)+2\alpha_M}$,
thus is finite dimensional. Therefore, there exists a finite positive integer $K$ such that
\[
X={\rm span}\{X^+_M(k)X^+_M(-k)v \mid 0< k<K\}.
\]
Thus  for any $r\in\Z$ we have
\begin{equation}\label{K}
X^+_M(r)X^+_M(-r)v=\sum\limits_{0< k<K}a^{(r)}_k X^+_M(k)X^+_M(-k)v,\quad a^{(r)}_k\in\C.
\end{equation}

Note that the elements $X^+_M(k)$ for all $k\in\Z$ anti-commute among themselves and satisfy
$X^+_M(k)^2=0$. Thus equation \eqref{M} below immediately follows from \eqref{K}.
\begin{lemma}\label{lem:v-vector}
Let $V$ be a simple zero-level integrable $\U_q$-module,  and
let $v\in V$ be a nonzero weight vector satisfying \eqref{0-hw}.
Then the following relations hold for large  $k$:
\begin{equation}\label{M}
X^+_M(n_k)X^+_M(-n_k)\cdots X^+_M(n_1)X^+_M(-n_1)v=0, \quad \forall n_1,\dots,n_k\in\Z;
\end{equation}
\begin{equation}\label{1M+N-1}
X^+_{1,M+N-1}(m_k)\cdots X^+_{1,M+N-1}(m_1)v=0, \quad \forall m_1, \dots, m_k\in\Z.
\end{equation}
\end{lemma}
\begin{proof} Since \eqref{M} was proven already, we only need to consider \eqref{1M+N-1}.
For notational simplicity, we write $E(m)=X^+_{1, M+N-1}(m)$ for all $m$.
Applying
\[(X^+_{M+N-1}(0))^{2k}\cdots(X^+_{M+1}(0))^{2k}(X^+_{1}(m))^{2k}(X^+_{2}(0))^{2k}\cdots (X^+_{M-1}(0))^{2k}\]
to \eqref{M} and then using  \eqref{ii}, we can show that
\begin{equation}\label{2-1M+N-1}
E(m+n_k)E(m-n_k)\cdots  E(m+n_1)E(m-n_1)v=0.
\end{equation}
Let $l+1$ be the minimal integer such that \eqref{2-1M+N-1} holds. Then there exist $r_1,\cdots,r_{l}$ such that
\begin{eqnarray*}
&&v':=E(m+r_{l})E(m-r_{l})\cdots  E(m+r_1)E(m-r_1)v\ne 0,\\
&&E(p)E(k)v'=0\quad \mbox{ for all }p,  k\in\Z\mbox{ with }p\equiv k~({\rm mod} ~2).
\end{eqnarray*}
By Lemma \ref{key-1}, we have
\[
X^+_i(m)(E(n_1)\cdots  E(n_k)v')=0, \quad i\ne M, n_1,\dots,n_k\in\Z, k\in\Z_{\ge 0}.
\] Now \eqref{1M+N-1} follows from \eqref{3-}.
\end{proof}

Using Lemma \ref{lem:v-vector}, we can prove the following result.
\begin{proposition}\label{prop:key}
Let $V$ be an irreducible zero-level integrable $\U_q$-module with finite dimensional weight spaces.
Then there always exists a nonzero weight vector $w\in V$ such that
\begin{eqnarray}
\label{semi-positive}
&&X^+_{i}(r)w=0,\quad \forall i\ne M, \  r\in\Z,\\
\label{semi-positive1}
&&X^+_{a,b}(r)w=0, \quad  \forall (a,b)\in S, \  r\in\Z, \\
\label{semi-positive2}
&&X^+_M(n_1)\cdots X^+_M(n_k)w=0, \quad  \forall n_i\in\Z,  \ \text{large $k$}.
\end{eqnarray}
\end{proposition}
\begin{proof}
By Lemma \ref{lem:v-vector}, one can find a non-zero weight vector $v_{1,M+N-1}$ such that
\[
X^+_i(m)v_{1,M+N-1}=X^+_{1,M+N-1}(m)v_{1,M+N-1}=0\quad i\ne M, m\in\Z.
\]
 We observe that \eqref{M} still holds if we replace $v$ by $v_{1,M+N-1}$, namely, for large $k$,
\begin{equation}\label{M-1M+N-1}
X^+_M(n_k)X^+_M(-n_k)\cdots X^+_M(n_1)X^+_M(-n_1)v_{1,M+N-1}=0, \quad \forall n_1,\dots,n_k\in\Z.
\end{equation}
Applying $(X^+_{M+N-2}(0))^{2k}\cdots(X^+_{M+1}(0))^{2k}(X^+_{1}(m))^{2k}(X^+_{2}(0))^{2k}\cdots (X^+_{M-1}(0))^{2k}$ to this equation, and then using  \eqref{ii},
we obtain
\[
X^+_{1,M+N-2}(m+n_k)X^+_{1,M+N-2}(m-n_k)\cdots  X^+_{1,M+N-2}(m+n_1)X^+_{1,M+N-2}(m-n_1)v=0.
\]
Clearly, $m+n_i\equiv m-n_i ~({\rm mod}~2), i=1,\dots, k$. From Lemma \ref{key-1} and Lemma \ref{key-2},
we can find a non-zero weight vector $v_{1,M+N-2}$ such that for all $m\in\Z$,
\begin{eqnarray*}
&&X^+_i(m)v_{1,M+N-2}=0,\quad i\ne M,\\
&&X^+_{1,M+N-1}(m)v_{1,M+N-2}=X^+_{1,M+N-2}(m)v_{1,M+N-2}=0.
\end{eqnarray*}

Repeating the above arguments for a finite number of times, we will find a nonzero weight vector $w$ such that
\begin{equation}\label{semi-positive0}
X^+_i(m)w= X^+_{a,b}(m)w=0, \quad i\ne M, \   m\in\Z, \ (a,b)\in S.
\end{equation}

Let $\mu$ be the weight of $w$.  Observe that $V$, being irreducible, must be cyclically generated by $w$ over $\U_q$.
By using the PBW theorem for  $\U_q$ and equation \eqref{semi-positive0},
we easily show that any weight of $V$ which is bigger than $\mu$ (relative to $B$; see also the Borel subalgebra of
$\widehat{\mathfrak{sl}}(M|N)$ defined by \eqref{eq:borel-2}) must be of the form
\begin{eqnarray}\label{eq:is-a-weight}
\mu +a(\es_M-\es_{M+1})+b\d, \quad a\in\Z_{\ge 0},\,b\in\Z.
\end{eqnarray}

Now we prove \eqref{semi-positive2}. Suppose it is false, that is,
for any positive integer $p$,
there always exist $k>p$ and $n_1, \dots, n_k\in \Z$ such that  $\tilde{w}:=X^+_M(n_1)\cdots X^+_M(n_k)w\ne 0$.
Then $\nu:=\mu+k(\es_M-\es_{M+1})+\sum\limits_{i=1}^kn_i\d$ is the weight of
$\tilde{w}$. But for large $p$, and hence large $k$, we have $(\nu,\, \es_{M-1}-\es_M)<0$. Thus $\nu+(\es_{M-1}-\es_M)$ is a weight of $V$ by considering the action of the $\U_q(\mathfrak{sl}_2)$ subalgebra generated by  $X^{\pm}_{M-1}(0)$ and  $K^{\pm 1}_{M-1}$.
However, the weight $\nu+(\es_{M-1}-\es_M)$ is not of the form \eqref{eq:is-a-weight}, proving \eqref{semi-positive2} by contradiction.
\end{proof}

The following theorem is now an easy consequence of Proposition \ref{prop:key}.
\begin{theorem}\label{h-vector}\label{thm:hw}
Let $V$ be an irreducible zero-level integrable $\U_q$-module with finite dimensional weight spaces.
Then $V$ is a highest weight module with respect to the triangular decomposition of
$\U_q$ given by \eqref{loop-tri-decomp}.
\end{theorem}
\begin{proof}
Consider the weight vector $w$ of Proposition \ref{prop:key},
and let $s$ be the minimal integer such that \eqref{semi-positive2} holds.
Then there exist $r_1,\dots,r_{s-1}\in\Z$ such that
\begin{eqnarray*}
&&v:=X^+_M(r_1)\cdots X^+_M(r_{s-1})w\ne 0,\\
&& X^+_M(r) v=0, \quad \forall r\in \Z.
\end{eqnarray*}
It it not difficult to show that we also have $X^+_i(m)v=0$ for all $i\ne M$ and $m\in\Z$.
\end{proof}

\begin{theorem} \label{thm:class}
Let $W$ be an irreducible integrable $\U_q$-module of type $1$ with finite dimensional weight spaces.
Then $W$ is isomorphic to an irreducible component of $L(V, b)$ for some $b\in\C$, where $V=V(P, f, c, Q)$
(see Definition \ref{def:fd-module}) for some $(P, f, c, Q)\in \mathcal{R}_{M, N}$.
\end{theorem}
\begin{proof}
It follows from Theorem \ref{thm:hw} that there exists a nonzero highest weight vector $v\in W$, and
$W=\U_q v=\U_q^-\U_q(0)v$, where the second equality follows from \eqref{loop-tri-decomp}.
Clearly $\U_q(0)v=\U'_q(0)v$.
Irreducibility of $W$ requires that $\U'_q(0)v$ be an irreducible $\U_q(0)$-module,
and hence an irreducible $\U'_q(0)$-module.
Since $\U'_q(0)$ is a $\Z$-graded commutative algebra, $\U'_q(0)v$ being an irreducible graded module must be the quotient of $\U'_q(0)$ by a maximal graded ideal $\mathcal{M}$ of $\U'_q(0)$ which annihilates $v$.  It follows from Lemma \ref{Cartan}
that $\U'_q(0)/\mathcal{M}\cong L_r:=\K[t^r, t^{-r}].$
Thus we have a natural $\Z$-graded homomorphism
$\tilde\varphi: \U'_q(0)\rightarrow \U'_q(0)/\mathcal{M}\cong L_r$ such that
$\tilde\varphi(x)v=x v$ for all $x\in \U'_q(0)$.
There exists some $b\in\C$ such that $Dv=q^bv$. Set $\varphi=(\tilde\varphi, b)$
(see notation immediately below \eqref{eq:T-mod}).
Then $W$ is isomorphic to $V(\varphi)$.

Set  $\psi=S\circ \tilde\varphi$ and consider the irreducible module $V(\psi)$.
By Theorem \ref{iso}, $V(\varphi)$ is isomorphic to an
irreducible component of $L(V(\psi),b)$.
Since  $V(\varphi)$ is an integrable  $\U_q$-module with finite dimensional weight spaces,
so is $V(\psi)$.  Thus it follows from Theorem \ref{thm:fd-module} that $V(\psi)$ is
isomorphic to $V(P, f, c, Q)$ (see Definition \ref{def:fd-module}) for some $(P, f, c, Q)\in \mathcal{R}_{M,N}$.
This completes the proof.
\end{proof}


\section{Integrable representations at nonzero levels}
In this section, highest and lowest weight $\Uq$-modules are defined relative to
the triangular decomposition of $\Uq$ given in Proposition \ref{tri-decomp}.
\subsection{}

The subalgebra of $\Uq$ generated by
$X^{\pm}_i(n), K^{\pm 1}_i, h_i(r), C^{\pm 1/2}$, $D^{\pm 1}$
(with $M+1\le i\le M+N-1, \ n, r\in\Z, \ r\ne 0$) is
$\U_{q^{-1}}(\widehat{\mathfrak{sl}}(N))$. Thus $C$ acts on any nontrivial
simple integrable highest weight $\U_{q^{-1}}(\widehat{\mathfrak{sl}}(N))$-module \cite{cp91, L93}
by the multiplication by $q^{-\ell}$ for some fixed $\ell>0$.

We have the following result.
\begin{proposition}\label{prop:dominant+}
\begin{enumerate}
\item Let $W$ be an integrable  $\U_{q^{-1}}(\widehat{\mathfrak{sl}}(N))$-module
with finite dimensional weight spaces.
Suppose that the center $C$ acts on $W$ by $(q^{-1})^r$ with $r\in\Z_{>0}$.
If $\lambda$ is a weight of $W$, then
there exists $K>0$ such that
\begin{center}
$\l+\a+k\d$ is not a weight of $W$ for all $k\ge K$ and $\a\in \D_0^2\cup\{0\}$.
\end{center}
\item
Let $V$ be an irreducible integrable module for $\Uq$  with finite dimensional weight spaces.
Suppose that the center $C$ acts on $V$ by $(q^{-1})^r$ with $r\in\Z_{>0}$.
Then for any $\l\in P(V)$ there exists $K>0$ such that
\[
\l+\a+k\d\notin P(V)\quad \mbox{ for all } k\ge K \mbox{ and  for all }\a\in \D_0^2\cup\{0\}.
\]
\end{enumerate}
\end{proposition}
\begin{proof}
Part (1) can be easily proved by adapting the proof of \cite[Theorem 1.10]{r03} to the present context.
We omit the details.

To prove part (2),  set $T=\K[(K_M^NK_{M+1}^{N-1}\cdots
K_{M+N-2}^{2}K_{M+N-1})^{\pm 1}, K_{M-1}^{\pm 1},\dots, K_{1}^{\pm 1}]$.
Observe that $T$ commutes with $\U_{q^{-1}}(\widehat{\mathfrak{sl}}(N))$.
Decompose $V$ into the direct sum of $T$-invariant subspaces. Each  $T$-invariant subspace is an integrable $\U_{q^{-1}}(\widehat{\mathfrak{sl}}(N))$-module with finite dimensional weight spaces. Now part (1) implies part (2).
\end{proof}

\begin{theorem}\label{thm:level}
Assume that  both $M$ and $N$ are greater than $1$.
Then there exists no integrable $\Uq$-module with finite dimensional weight spaces,
where $C$ does not act by the identity.
\end{theorem}

\begin{proof}  Without lose of generality, we may assume that $C$ acts by $q^{-r}$ with $r>0$.
Let $\U_{q}(\widehat{\mathfrak{sl}}(M))$ be the subalgebra of $\Uq$ generated by
$X^{\pm}_i(n), K^{\pm 1}_i, h_i(s), C^{\pm1/2}$ with $1\le i\le M-1$, $n, s\in\Z$ and $s\ne 0$. Regard $V$ as an integrable
$\widehat{\U}_{q}(\mathfrak{sl}(M))$-module. Note that $C$ acts on $V$ by $q^{-r}$ with $r>0$.
By \cite[Theorem 5]{cg03}, there exists a weight vector $v\in V$ of weight $\l$ such that
$X^{\pm}_i(n)v=0$ and $h_i(n)v=0$ for all $n<0$ and $1\le i\le M-1$.
Then $h_i(n)v\ne 0$  for all $n>0$ and $1\le i\le M-1$. Thus $\l+n\d\in P(V)$ for all $n>0$.
This contradicts Proposition \ref{prop:dominant+}, completing the proof.
\end{proof}

\begin{remark}
A similar result has long been \cite{JZ, kw01} for affine Lie superalgebras in the classical setting.
\end{remark}

\subsection{}From now on we assume that $N>M=1$.
\begin{lemma}\label{lem:anti-comm}
$X^{\pm}_{1,a}(m)X^{\pm}_{1,a}(n)=-X^{\pm}_{1,a}(n)X^{\pm}_{1,a}(m), \quad 1\le a\le N-1, m,n\in\Z$.
\end{lemma}
\begin{proof}  The $a=1$ case is a defining relation of $\Uq$. For $a\ge 2$, we have
\begin{eqnarray*}
{[}X^{+}_{1,a}(m),X^{+}_{1,a}(n)]_{q^{-2}}
=[[X^{+}_{1}(m),X^+_{2,a}(0)]_{q^{-1}},[X^{+}_{1}(n),X^+_{2,a}(0)]_{q^{-1}}]_{q^{-2}}.
\end{eqnarray*}
We can rewrite the right hand side as
\begin{eqnarray*}
\begin{aligned}
{[}X^+_1(m),[X^+_{2,a}(0),X^+_{1,a}(n)]_{q^{-1}}]_{q^{-2}}
+q^{-1}[[X^+_1(m),[X^+_1(n),X^+_{2,a}(0)]_{q^{-1}}]_{q^{-1}},X^+_{2,a}(0)],
\end{aligned}
\end{eqnarray*}
where the first term vanishes by Lemma \ref{lem:relat}. The second term can be expressed as
\begin{eqnarray*}
\begin{aligned}
&q^{-1}[[[X^+_1(m),X^+_1(n)],X^+_{2,a}]_{q^{-2}},X^+_{2.a}(0)]\\
&-q^{-1}[[X^+_1(n),[X^+_1(m),X^+_{2,a}(0)]_{q^{-1}}]_{q^{-1}},X^+_{2.a}(0)],
\end{aligned}
\end{eqnarray*}
where the first term vanishes, as $[X^+_1(m),X^+_1(n)]=0$. By manipulating the second term, we obtain
\begin{eqnarray*}
\begin{aligned}
{[}X^{+}_{1,a}(m),X^{+}_{1,a}(n)]_{q^{-2}}=&-q^{-1}([X^+_1(n),[X^+_{1,a}(m),X^+_{2,a}(0)]_q]_{q^{-2}}\\
&+q[[X^+_1(n),X^+_{2,a}(0)]_{q^{-1}},[X^+_1(m),X^+_{2,a}(0)]_{q^{-1}}]_{q^{-2}}\\
=&-[X^{+}_{1,a}(n),X^{+}_{1,a}(m)]_{q^{-2}}.
\end{aligned}
\end{eqnarray*}
Hence, $X^+_{1,a}(m)X^+_{1,a}(n)=-X^+_{1,a}(n)X^+_{1,a}(m)$.

Similarly, one can show that $X^-_{1,a}(m)X^-_{1,a}(n)=-X^-_{1,a}(n)X^-_{1,a}(m)$.
\end{proof}
\begin{theorem}\label{thm:sl1-N}
Assume that $N>M=1$. Let $V$ be an irreducible integrable $\Uq$-module with finite dimensional weight spaces.
Suppose that $C$ acts by $(q^{-1})^r$ for some non-zero $r\in\Z$.
If $r>0$ (resp. $r<0$), then $V$ is a highest (resp. lowest) weight module.
\end{theorem}

\begin{proof}
Without lose of generality, we may assume that $r>0$.
\begin{claim}\label{claim:fd+}
 For any weight vector $v\in V$, the  following vector space, spanned by
\[
\left\{X^+_{1,a_1}(m_1)\cdots X^+_{1,a_k}(m_k)v\ \left| \
\begin{array}{l} 1\le a_1\le \cdots \le a_k\le N,  k\ge 0,  \\
m_i\ge 0,  m_i<m_{i+1}
\mbox{ when }a_i=a_{i+1}
\end{array}\right.
\right\},
\]
is finite-dimensional.
\end{claim}

By Proposition \ref{prop:PBW},  it is sufficient to prove that, for $1\le p\le N$,
the vector space $S^+_p(v)$ spanned by
$
\{X^+_{1,p}(m_1)\cdots X^+_{1,p}(m_r)v \ | \ r\in\Z_{\ge 0}, m_i\ge 0\}
$
is finite-dimensional.

For $S^+_1(v)$, which is spanned by
$
\{X^{+}_{1}(m_1)\cdots X^{+}_{1}(m_r)v\ | \ 0\le m_1<\cdots <m_r, r\in\Z_{\ge 0}\},
$
we consider
\begin{eqnarray*}
X^+_1(n)v&=&\frac{n(q-q^{-1})}{q^{-n}-q^n}C^{\frac{|n|}{2}}[h_{2}(n), X^+_1(0)]v\\
&=&\frac{n(q-q^{-1})}{q^{-n}-q^n}C^{\frac{|n|}{2}}\left(h_{2}(n)X^+_1(0)v-X^+_1(0)h_{2}(n)v\right).
\end{eqnarray*}
From Proposition \ref{prop:dominant+}  there exists $n_0>0$ such that $h_{2}(n)v=0$ and
$h_{2}(n)X^+_1(0)v=0$ for all   $n>n_0$.
Now it is easy to see that
$S^+_1(v)$ is spanned by
\[
\{X^{+}_{M}(m_1)\cdots X^{+}_{M}(m_r)v\ |\ r\ge 0, 0\le m_i\le n_0, m_i\ne m_j, i\ne j\},
\]
which is clearly finite-dimensional.

For any $r\ge 1$ and $m_i\ge 0$, observe that
\begin{equation}\label{fd-1}
X^{+}_{1}(m_1)\cdots X^{+}_{1}(m_r)v=0,
\end{equation}
if there exists $j (1\le j\le r)$ such that $m_j>n_0$. From Proposition \ref{prop:dominant+}
there exists $K_v>0$ such that $X^+_i(k)v=0$ for all $ k\ge K_v, 2\le i\le N$.
Applying
$(X^+_p(K_v))^r\cdots (X^+_{M-1}(K_v))^r$
to
\eqref{fd-1},
and using \eqref{ii} \eqref{ii+1+} repeatedly, we have
\[
X^+_{1,p}(m_1+(M-p)K_v)
\cdots X^+_{1,p}(m_r+(M-p)K_v)v=0.
\]
if there exists $j (1\le j\le r)$ such that $m_j>n_0$.
Combining this with Lemma \ref{lem:anti-comm}, we conclude that
the vector space $S^+_p(v)$ is spanned by
\[
\{X^+_{1,p}(m_1)\cdots X^+_{1,p}(m_r)v\ |\ 0\le m_1<\cdots<m_{r}\le (M-p)K_v+n_0, r\in\Z_{\ge 0}\},
\]
which is finite-dimensional. This completes the proof of  Claim 1.

In a similar way, one can prove
\begin{claim}\label{claim:fd-}
 For any weight vector $v\in V$, the vector space spanned by the following set
\[
\left\{X^-_{1,a_1}(m_1)\cdots X^-_{1,a_k}(m_k)v \left|
\begin{array}{l}  1\le a_1\le \cdots \le a_k\le N, k\ge 0,\\
m_i> 0,  m_i<m_{i+1}\mbox{ when }a_i=a_{i+1}
\end{array}\right.
\right\},
\]
is finite-dimensional.
\end{claim}

Let $\mathcal{N}^+$ (resp. $\mathcal{N}^-$) be the subalgebra of $\Uq$ generated by
 $X^+_{1,a}(0), X^{\pm }_{1,a}(n),  n>0, 1\le a\le N$ (resp.  $ X^+_{1,a}(0), X^{\pm }_{1,a}(n),  n<0, 1\le a\le N$).
Combining Proposition \ref{prop:PBW} with Claims 1 and 2, we obtain
\begin{claim}\label{claim:fd}
For any weight vector $v\in V$, the space $\mathcal{N}^+v$ is finite-dimensional.
\end{claim}

For any weight vector $v\in V$, set
$
W=\K[K_1,K^{-1}_1] \U\left(\bigoplus\limits_{n>0}\K h_1(n)\right) \mathcal{N}^+ v$.   From
Proposition \ref{prop:dominant+} and Claim \ref{claim:fd}, one can see that
$W$ is finite-dimensional.

Define the following subalgebras of $\Uq$:
\begin{itemize}
\item $\U_{q^{-1}}(\widehat{\mathfrak{sl}}(N))$ is generated by
$X^{\pm}_i(n), h_i(r), K^{\pm 1}_i,  C^{\pm 1/2}, D^{\pm 1},$ $n,r\in\Z, r\ne 0, 2\le i\le N$;
\item ${\widehat\U}^+_{q^{-1}, N}$ by $\{ X^+_i(0), X^{\pm}_i(n),  h_i(n)  \mid n>0, 2\le i\le N\}$;
\item ${\widehat\U}^-_{q^{-1}, N}$ by $\{X^-_i(0), X^{\pm}_i(n), h_i(n) \mid n<0, 2\le i\le N\}$; and
\item ${\widehat\U}^0_{q^{-1}, N}$
by $\{K^{\pm 1}_i,  C^{\pm 1/2}, D^{\pm 1}\mid 2\le i\le N\}$.
\end{itemize}
Now consider
$
\mathcal{W}={\widehat\U}^-_{q^{-1}, N}{\widehat\U}^0_{q^{-1}, N}
{\widehat\U}^+_{q^{-1}, N}W.
$
Clearly,
 $\mathcal{W}$ is an integrable $\U_{q^{-1}}(\widehat{\mathfrak{sl}}(N))$-module.
Using proposition \ref{prop:dominant+}, one can show that
${\widehat\U}^+_{q^{-1}, N}W$ is finite-dimensional.
By \cite[Proposition 1.7]{cg03}, we have
$
\mathcal{W}\cong \bigoplus_\l m_\l V(\l),
$
where $V(\l)$ are irreducible integrable $\U_{q^{-1}}(\widehat{\mathfrak{sl}}(N))$-modules with highest weight $\l$
and multiplicities $m_\l\in\Z_{\ge 0}$, which are nonzero
for only finitely many $\l$. Thus $\mathcal{W}$ has a maximal weight.
Since $V$ is simple, $V=\mathcal{N}^-\U\left(\bigoplus\limits_{n<0}\K h_1(n)\right)\mathcal{W}$. Thus
this maximal weight of $\mathcal{W}$ is also the highest weight of $V$.
\end{proof}

Recall that for ordinary quantum affine algebras,
a simple highest (resp. lowest) weight module is integrable if and only if its
highest (resp. lowest) weight is integral dominant (resp. anti-dominant)  \cite{cp91, L93}.

\begin{corollary}\label{cro:criterion}
Assume that $N>M=1$. A simple $\Uq$-module $V$ at nonzero level is integrable
with finite dimensional weight spaces if and only if $V$ is
\begin{enumerate}
\item  a highest weight module
with a highest weight which is integral dominant with respect to $\U_q(\widehat{\mathfrak{sl}}(M|N)_{\bar 0})$,  or
\item  a lowest weight module
with a lowest weight which is integral anti-dominant
with respect to $\U_q(\widehat{\mathfrak{sl}}(M|N)_{\bar 0})$.
\end{enumerate}
\end{corollary}
\begin{proof}
As simple highest or lowest weight $\Uq$-modules defined with respect
to the triangular decomposition of $\Uq$ given in Proposition \ref{tri-decomp}
automatically have finite dimensional weight spaces, the corollary
immediately follows from Theorem \ref{thm:sl1-N} and the  preceding remarks
on integrable highest weight modules for ordinary quantum affine algebras.
\end{proof}


\appendix
\section{Relations in $\U_q$}
We present some technical results which are used in the main body of the paper.

The following identities are valid for the $q$-bracket.
\begin{lemma}
For any homogeneous elements $a, b, c$ of $\Uq$, and nonzero scalars $u, v, x$,
\begin{eqnarray}\label{q-bracket}
\begin{aligned}
&[a, bc]_v=[a,b]_xc+(-1)^{|a||b|}x b[a,c]_{\frac{v}{x}},\\
&[a b, c]_v=a[b,c]_x+(-1)^{|b||c|}x[a,c]_{\frac{v}{x}}b,\\
&[a, [b, c]_u]_v=[[a,b]_x, c]_{\frac{uv}{x}}+(-1)^{|a||b|}x[b,[a,c]_{\frac{v}{x}}]_{\frac{u}{x}},\\
&[[a, b]_u, c]_v=[a,[b,c]_x]_{\frac{uv}{x}}+(-1)^{|b||c|}x[[a,c]_{\frac{v}{x}},b]_{\frac{u}{x}}.
\end{aligned}
\end{eqnarray}
\end{lemma}

We can derive from Definition \ref{def:qaff} the following relations:
\begin{eqnarray}
&&\label{ii}
\begin{aligned}
&[[X^+_i(m),[X^+_i(m),X^+_j(k)]_{q^{-1}}]_q\\
&=[[X^+_i(m),[X^+_i(m),X^+_j(k)]_q]_{q^{-1}}=0, \quad i\ne M, \  a_{i,j}=\pm 1,
\end{aligned}\\
&&\label{MM}
\begin{aligned}
&[[X^+_M(m),[X^+_M(m),X^+_j(k)]_{q^{-1}}]_{q^{-1}}\\
&=[X^+_M(m),[X^+_M(m),X^+_j(k)]_q]_q=0, \quad a_{M,j}=\pm 1.
\end{aligned}
\end{eqnarray}
\begin{eqnarray}\label{ii+1+}
[X^+_{i- 1}(m),X^+_{i}(n)]_{q_{i}} &=& q_{i}^k[X^+_{i- 1}(m+k),X^+_{i}(n-k)]_{q_{i}}\\\nonumber
 &+ &\sum\limits_{s=1}^k q_{i}^{s-1}(q_{i}^2-1)X^+_{i}(n-s)X^+_{i- 1}(m+s),
\end{eqnarray}
\begin{eqnarray}\label{ii+1-}
[X^+_{i-1}(m),X^+_{i}(n)]_{q_{i}} &=& q_{i}^{-k}[X^+_{i- 1}(m-k),X^+_{i}(n+k)]_{q_{i}}\\\nonumber
& -&  \sum\limits_{s=0}^{k-1}q_{i}^{-s-1}(q_{i}^2-1)X^+_{i}(n+s)X^+_{i- 1}(m-s).
\end{eqnarray}
\begin{eqnarray}\label{MM+1b}
[X^+_M(m),X^+_{M+1}(n)]_q &= & q^k[X^+_{M}(m+k),X^+_{M+1}(n-k)]_q\\\nonumber
& + & \sum\limits_{s=0}^{k-1}q^{s}(1-q^2)X^+_{M}(m+s)X^+_{M+1}(n-s).
\end{eqnarray}
\begin{eqnarray}\label{>M+}
[X^+_{i+1}(m),X^+_{i}(n)]_{q} &=& q^k[X^+_{i+1}(m-k),X^+_{i}(n+k)]_{q}\\\nonumber
&+ &\sum\limits_{s=1}^k q^{s-1}(q^2-1)X^+_{i}(n+s)X^+_{i+1}(m-s),\quad i\ge M.
\end{eqnarray}
\begin{eqnarray}\label{>M-}
[X^+_{i+1}(m),X^+_{i}(n)]_{q} &=& q^{-k}[X^+_{i+1}(m+k),X^+_{i}(n-k)]_{q}\\\nonumber
& -&  \sum\limits_{s=0}^{k-1}q^{-s-1}(q^2-1)X^+_{i}(n-s)X^+_{i+1}(m+s),\quad i\ge M.
\end{eqnarray}

Combining \eqref{ii+1+} with \eqref{ii+1-},  \eqref{>M+} with \eqref{>M-}, respectively, we have

\begin{eqnarray}\label{ii+1}
[X^+_{i-1}(m),X^+_{i}(n)]_{q_{i}} &=& q_{i}^{\pm k}[X^+_{i- 1}(m\pm k),X^+_{i}(n\mp k)]_{q_{i}}\\\nonumber
& +&  \sum\limits_{s=0}^{|k|}c_sX^+_{i}(n\mp s)X^+_{i- 1}(m\pm s)\quad c_s\in \K,
\end{eqnarray}
\begin{eqnarray}
\label{>M}
[X^+_{i+1}(m),X^+_{i}(n)]_{q} &=& q^{\pm k}[X^+_{i+1}(m\mp k),X^+_{i}(n\pm k)]_{q}\\\nonumber
& +&  \sum\limits_{s=0}^{|k|}c_sX^+_{i}(n\pm s)X^+_{i+1}(m\mp s),\quad i\ge M, c_s\in \K.
\end{eqnarray}

We have the following result.
\begin{lemma}\label{lem:relat}
\begin{enumerate}
\item $[[[X^+_{i-1}(m),X^+_i(n)]_{q_{i}},X^+_{i+1}(k)]_{q_{i+1}},X^+_i(n)]=0,\quad i\ne M$.

\item $[X^+_i(0), X^+_{a,b}(n)]=0, \quad a<i<b, n\in\Z.$

\item $[X^+_M(0), X^+_{a,M}(n)]_{q^{-1}}=0, \quad a<M, n\in\Z.$

\item $[X^+_b(0), X^+_{a,b}(n)]_{q_b}=0, \quad b\ne M, n\in\Z.$
\end{enumerate}
\end{lemma}
\begin{proof}
Part (1) can be found in \cite[Lemma 6.1.1]{y94}. Part (2) follows from (1),
part (3) follows from \eqref{MM}, and part (4)  follows from \eqref{ii}.
\end{proof}

\section{Proof of Lemma \ref{key-1}}\label{pf:key-1}
\begin{proof}[Proof of Lemma \ref{key-1}]
Set $v^{(t-1)}_{a,b}:=X^+_{a,b}(n_{t-1})\cdots X^+_{a,b}(n_1)v_{a-1,b-1}, t=1,\dots,p+1. $
We use induction on $t$ starting from the given case $t=1$. Assume that
\begin{equation}\label{eq:induct}
X^+_{i}(m)v^{(t-1)}_{a,b}=X^+_{k,l}(m)v^{(t-1)}_{a,b}=0 \mbox{ for all } i\ne M,  m\in \Z, (k,l)> (a,b).
\end{equation}

First we want to prove that $X^+_{i}(m)v^{(t)}_{a,b}=0$ for all $i\ne M.$

If $i>b+1$ or $i<a-1$, then $[X^+_i(m),X^+_{a,b}(n)]=0$. Thus $X^+_i(m)v^{(t)}_{a,b}=0$.

If $i=a-1$, we have
\begin{eqnarray*}
X^+_{a-1}(m)v^{(t)}_{a,b} &=&  [X^+_{a-1}(m), X^+_{a,b}(n_{t})]_{q_{a}}v^{(t-1)}_{a,b}\\
&=&  [[X^+_{a-1}(m), X^+_{a}(n_{t})]_{q_{a}},X^+_{a+1,b}(0)]_{q_{a+1}}v^{(t-1)}_{a,b}.
\end{eqnarray*}
Using equation \eqref{ii+1}, we can rewrite the right hand side as
\begin{eqnarray*}
\begin{aligned}
&q_a^{ n_{t}}[[X^+_{a-1}(m+n_{t})), X^+_{a}(0)]_{q_{a}},X^+_{a+1,b}  (0)]_{q_{a+1}}v^{(t-1)}_{a,b}\\
&+\sum\limits_{s=0}^{|n_{t}|} c_s[[X^+_{a}(n_t\mp s)X^+_{a-1}(m\pm s),X^+_{a+1,b}(0)]_{q_{a+1}}v^{(t-1)}_{a,b}\\
&= q_a^{n_{t}}X^+_{a-1,b}(m+n_{t})v^{(t-1)}_{a,b}+\sum\limits_{s=0}^{|n_{t}|}c_sX^+_{a,b}(n_{t}\mp s)X^+_{a-1}(m\pm s)v^{(t-1)}_{a,b}\\
&= 0.
\end{aligned}
\end{eqnarray*}

If $i=b+1$, we consider the case with $a=M-1$ and $b=M+1$ as an example,
and the proof for the general case is similar. By \eqref{>M+} and \eqref{>M-}, we have
\begin{eqnarray*}
\begin{aligned}
&X^+_{M+2}(m)v^{(t)}_{M-1,M+1}=X^+_{M+2}(m)X^+_{M-1,M+1}(n_t)v^{(t-1)}_{M-1,M+1}\\
&=[X^+_{M+2}(m), [X^+_{M-1,M}(n_t), X^+_{M+1}(0)]_{q^{-1}}]_q v^{(t-1)}_{M-1,M+1}\\
&=[X^+_{M-1,M}(n_t), [X^+_{M+2}(m),X^+_{M+1}(0)]_q]_{q^{-1}}v^{(t-1)}_{M-1,M+1}\\
&=q^{m}[X^+_{M-1,M}(n_t), [X^+_{M+2}(0),X^+_{M+1}(m)]_q]_{q^{-1}}v^{(t-1)}_{M-1,M+1}\\
&+\sum\limits_{s=0}^{|m|} c_s    [X^+_{M-1,M}(n_t), X^+_{M+1}(\pm s)]_{q^{-1}} X^+_{M+2}(m\mp s)v^{(t-1)}_{M-1,M+1},
\end{aligned}
\end{eqnarray*}
where the second term on the right hand side vanishes by \eqref{eq:induct}.  We can rewrite first term as
$-q^{1+m}[[[X^+_{M-1}(n_t), X^+_M(0)]_q,
X^+_{M+1}(m)]_{q^{-1}}, X^+_{M+2}(0)]_{q^{-1}}v^{(t-1)}_{M-1,M+1}$, which, by \eqref{q-bracket}, is equal to
\begin{eqnarray*}
\begin{aligned}
&-q^{1+m}[[[X^+_{M-1}(n_t), [X^+_M(0),X^+_{M+1}(m)]_{q^{-1}}]_q, X^+_{M+2}(0)]_{q^{-1}}v^{(t-1)}_{M-1,M+1}\\
&=q^{m}[[[X^+_{M-1}(n_t), [X^+_{M+1}(m),X^+_M(0)]_{q}]_q, X^+_{M+2}(0)]_{q^{-1}}v^{(t-1)}_{M-1,M+1}.
\end{aligned}
\end{eqnarray*}
Using \eqref{>M}, we can cast the right hand side into
\begin{eqnarray*}
\begin{aligned}
&q^{2m}[[[X^+_{M-1}(n_t), [X^+_{M+1}(0),X^+_M(m)]_{q}]_q, X^+_{M+2}(0)]_{q^{-1}}v^{(t-1)}_{M-1,M+1}\\
&+\sum\limits_{k=0}^{|m|} c'_k[[[X^+_{M-1}(n_t), X^+_M(m\pm k)X^+_{M+1}(\mp k)]_q, X^+_{M+2}(0)]_{q^{-1}}v^{(t-1)}_{M-1,M+1},
\end{aligned}
\end{eqnarray*}
where the second term vanishes by \eqref{eq:induct}, and the first term can be rewritten as
\[
-q^{1+2m}[[[X^+_{M-1}(n_t), X^+_M(m)]_q,X^+_{M+1}(0)]_{q^{-1}},
X^+_{M+2}(0)]_{q^{-1}}v^{(t-1)}_{M-1,M+1}.
\]
By \eqref{ii+1}, this can be expressed as
\begin{eqnarray*}
\begin{aligned}
&-q^{1+3m}[[[X^+_{M-1}(n_t+m), X^+_M(0)]_q,X^+_{M+1}(0)]_{q^{-1}}, X^+_{M+2}(0)]_{q^{-1}}v^{(t-1)}_{M-1,M+1}\\
&+\sum\limits_{l=0}^{|m|} c''_l[[X^+_{M}(m\pm l) X^+_{M-1}(n_t\mp l),X^+_{M+1}(0)]_{q^{-1}}, X^+_{M+2}(0)]_{q^{-1}}v^{(t-1)}_{M-1,M+1}\\
&=-q^{1+3m}X^+_{M-1,M+2}(n_t+m)v^{(t-1)}_{M-1,M+1}=0.
\end{aligned}\\
\end{eqnarray*}

For $i=a (a\ne M)$, we obviously have
\begin{eqnarray*}
\begin{aligned}
&X^+_{a}(m)v^{(t)}_{a,b}= X^+_{a}(m) X^+_{a,b}(n_{t})v^{(t-1)}_{a,b}= [X^+_{a}(m), X^+_{a,b}(n_{t})]_{q^{-1}_{a+1}}v^{(t-1)}_{a,b}\\
&=[[[X^+_{a}(m), [X^+_{a}(n_{t}),X^+_{a+1}(0)]_{q_{a+1}}]_{q^{-1}_{a+1}},
X^+_{a+2,b}(0)]_{q_b}v^{(t-1)}_{a,b},
\end{aligned}
\end{eqnarray*}
which can be rewitten as
\begin{eqnarray*}
\begin{aligned}
& q_a^{m-n_{t}}
[[[X^+_{a}(m), [X^+_{a}(m),X^+_{a+1}(n_{t}-m)]_{q_{a+1}}]_{q^{-1}_{a+1}}, X^+_{a+2,b}(0)]_{q_b}v^{(t-1)}_{a,b}\\
&+\sum\limits_{s=0}^{|n_{t}-m|}c_s [[[X^+_{a}(m), X^+_{a+1}(\pm s)X^+_{a}(n_{t}\mp s)]_{q^{-1}_{a+1}}, X^+_{a+2,b}(0)]_{q_b}v^{(t-1)}_{a,b},
\end{aligned}
\end{eqnarray*}
by using \eqref{ii+1}.  The first term vanishes by \eqref{ii}, and the second terms is equal to
\begin{eqnarray*}
\sum\limits_{s=0}^{|n_{t}-m|}c_s\left(X^+_{a}(m)X^+_{a+1,b}
(\pm s)X^+_{a}(n_{t}\mp s)-{q^{-1}_{a+1}}X^+_{a+1,b}(\pm s)X^+_{a}(n_{t}\mp s)X^+_{a}(m)\right)v^{(t-1)}_{a,b},
\end{eqnarray*}
which obviously vanishes.
Similarly, one can prove that $X^+_b(m)v^{(t)}_{a,b}=0$.

For $a<i< M$,
\begin{eqnarray*}
\begin{aligned}
&X^+_{i}(m)v^{(t)}_{a,b}= [X^+_{i}(m),X^+_{a,b}(n_{t})]v^{(t-1)}_{a,b}\\
&=  [X^+_{a,i-2}(n_{t}),[ [X^+_i(m),[X^+_{i-1}(0),[X^+_{i}(0),X^+_{i+1}(0)]_{q}]_{q}], X^+_{i+2,b}(0)]_{q_b}]_q v^{(t-1)}_{a,b}\\
& =  q_i^{m}[X^+_{a,i-2}(n_{t}),[[X^+_i(m),[X^+_{i-1}(0),[X^+_{i}(m),X^+_{i+1}(-m)]_{q}]_{q}], X^+_{i+2,b}(0)]_{q_b}]_qv^{(t-1)}_{a,b}\\
& +\sum\limits_{s=0}^{|m|}c_s[X^+_{a,i-2}(n_{t}),[[X^+_i(m),[X^+_{i-1}(0),X^+_{i+1}(\pm s)X^+_{i}(\mp s)]_{q}], X^+_{i+2,b}(0)]_{q_b}]_qv^{(t-1)}_{a,b},
\end{aligned}
\end{eqnarray*}
where the first term on the right hand side vanishes by lemma \ref{lem:relat} (1).  Hence
\begin{eqnarray*}
X^+_{i}(m)v^{(t)}_{a,b}
&=&\sum\limits_{s=0}^{|m|}c_s [X_i^+(m),[X^+_{a,i-1}(n_{t}), X^+_{i+1,b}(\pm s)X^+_{i}(\mp s)]_{q_i} ]     v^{(t-1)}_{a,b}\\
&=&0.
\end{eqnarray*}
 Similarly, one can prove that $X^+_{i}(m)v^{(t)}_{a,b}=0$  for $M<i<b$.

Thus we have proved that $X^+_i(m)v^{(t)}_{a,b}=0$, $\forall i\ne M, m\in\Z.$

Now we prove $X^+_{k,l}(m)v^{(t)}_{a,b}=0 \mbox{ for all } m\in \Z, (k,l)> (a,b)$.

 For $k<a<M, l< b$,
\begin{eqnarray*}
\begin{aligned}
&X^+_{k,l}(m)v^{(t)}_{a,b}= X^+_{k,l}(m)X^+_{a,b}(n_t)v^{(t-1)}_{a,b}
=  [[X^+_{k,a}(m),X^+_{a+1,l}(0)]_{q},X^+_{a,b}(n_t)]v^{(t-1)}_{a,b}\\
&=[[X^+_{k,a}(m),[X^+_{a+1,l}(0),X^+_{a,b}(n_t)]]_qv^{(t-1)}_{a,b}
-[[X^+_{k,a}(m),X^+_{a,b}(n_t)],X^+_{a+1,l}(0)]_{q}v^{(t-1)}_{a,b},
\end{aligned}
\end{eqnarray*}
where the first term on the right hand side vanishes by Lemma \ref{lem:relat}. It is not difficult to show that $[X^+_{k,a}(m),X^+_{a,b}(n_t)]v^{(t-1)}_{a,b}=0$. Hence we can rewrite the right hand side as
\begin{eqnarray*}
\begin{aligned}
&-[X^+_{k,a}(m),X^+_{a,b}(n_t)]X^+_{a+1,l}(0)v^{(t-1)}_{a,b}\\
&=-[[X^+_{k,a-1}(m), X^+_a(0)]_{q}, [X^+_{a}(n_t), X^+_{a+1,b}(0)]_{q}]X^+_{a+1,l}(0)v^{(t-1)}_{a,b}\\
&=-[[[X^+_{k,a-1}(m),X^+_a(0)]_{q}, X^+_a(n_t)]_{q^{-1}}, X^+_{a+1,b}(0)]_{q^2}X^+_{a+1,l}(0)v^{(t-1)}_{a,b}\\
&-q^{-1}[X^+_a(n_t), [[X^+_{k,a-1}(m),X^+_a(0)]_{q}, X^+_{a+1,b}(0)]_{q}]_{q^2}X^+_{a+1,l}(0)v^{(t-1)}_{a,b}\\
&=-[[ [[X^+_{k,a-2}(m),X^+_{a-1}(0)]_{q},X^+_a(0)]_{q},  X^+_{a}(n_{t})]_{q^{-1}}, X^+_{a+1,b}(0)]_{q^{2}}X^+_{a+1,l}(0)v^{(t-1)}_{a,b}\\
&-q^{-1}[X^+_a(n_t), X^+_{k,b}(m)]_{q^2}X^+_{a+1,l}(0)v^{(t-1)}_{a,b},
\end{aligned}
\end{eqnarray*}
where the second term on the right hand side vanishes, and the first term can be rewritten as
\begin{eqnarray*}
\begin{aligned}
&-[[ X^+_{k,a-2}(m),[[X^+_{a-1}(0),X^+_a(0)]_{q},  X^+_{a}(n_{t})]_{q^{-1}} ]_{q},X^+_{a+1,b}(0)]_{q^{2}}X^+_{a+1,l}(0)v^{(t-1)}_{a,b}\\
=&-q^{-n_t}[[ X^+_{k,a-2}(m),[[X^+_{a-1}(-n_t),X^+_a(n_t)]_{q},  X^+_{a}(n_{t})]_{q^{-1}} ]_{q},X^+_{a+1,b}(0)]_{q^{2}}X^+_{a+1,l}(0)v^{(t-1)}_{a,b}\\
&+\sum\limits_{s=0}^{|n_t|}c_s[[ X^+_{k,a-2}(m),[X^+_{a}(\pm s)X^+_{a-1}(\mp s),  X^+_{a}(n_{t})]_{q^{-1}} ]_{q},X^+_{a+1,b}(0)]_{q^{2}}X^+_{a+1,l}(0)v^{(t-1)}_{a,b}.
\end{aligned}
\end{eqnarray*}
The first term on the right hand side vanishes by \eqref{ii}, and the second term can be expanded into
\begin{eqnarray*}
&&\sum\limits_{s=0}^{|n_t|}c_s [X^+_{a}(\pm s)[ X^+_{k,a-2}(m),X^+_{a-1}(\mp s)]_{q},  X^+_{a,b}(n_{t})]X^+_{a+1,l}(0)v^{(t-1)}_{a,b}\\
&&+q^2\sum\limits_{s=0}^{|n_t|}c_s [X^+_{a,b}(\pm s)[ X^+_{k,a-2}(m),X^+_{a-1}(\mp s)]_{q},  X^+_{a}(n_{t})]_{q^{-2}}X^+_{a+1,l}(0)v^{(t-1)}_{a,b}
\end{eqnarray*}
by using \eqref{q-bracket}. We can show that the first term vanishes identically, and the second can be rewritten as
\begin{eqnarray*}
\begin{aligned}
&q^2\sum\limits_{s=0}^{|n_t|}c_s X^+_{a,b}(\pm s)X^+_{k,a-2}(m)X^+_{a-1}(\mp s) X^+_{a}(n_{t})X^+_{a+1,l}(0)v^{(t-1)}_{a,b}\\
&=q^2\sum\limits_{s=0}^{|n_t|}c_s X^+_{a,b}(\pm s)X^+_{k,a-2}(m)X^+_{a-1}(\mp s) X^+_{a,l}(n_{t})v^{(t-1)}_{a,b}\\
&=q^2\sum\limits_{s=0}^{|n_t|}c_s X^+_{a,b}(\pm s)X^+_{k,l}(m\mp s+n_{t})v^{(t-1)}_{a,b}\\
& =0.
\end{aligned}
\end{eqnarray*}

By modifying the above computations slightly,
one can prove that  $X^+_{k,l}(m)v^{(t)}_{a,b}=0$ for $k<a=M, l< b.$
It is even easier to show that $X^+_{k,l}(m)v^{(t)}_{a,b}=0$ for $k<a, l\ge b.$

Now consider the cases $k\ge a, l>b$.

For $k<M$, we have
\begin{eqnarray*}
\begin{aligned}
&X^+_{k,l}(m)v^{(t)}_{a,b}= X^+_{k,l}(m)X^+_{a,b}(n_t)v^{(t-1)}_{a,b}\\
&=[X^+_k(m), [X^+_{k+1,b}(0),X^+_{b+1,l}(0)]_{q^{-1}}]_{q_{k+1}}X^+_{a,b}(n_t)v^{(t-1)}_{a,b}\\
&=X^+_k(m)[X^+_{k+1,b}(0),X^+_{b+1,l}(0)]_{q^{-1}}X^+_{a,b}(n_t)v^{(t-1)}_{a,b}\\
&=(X^+_k(m)X^+_{k+1,b}(0)X^+_{b+1,l}(0)-q^{-1}X^+_k(m)X^+_{b+1,l}(0)X^+_{k+1,b}(0))X^+_{a,b}(n_t)v^{(t-1)}_{a,b},
\end{aligned}
\end{eqnarray*}
which, by using  Lemma \ref{lem:relat},  can be expressed as
\begin{eqnarray*}
&&-qX^+_k(m)X^+_{k+1,b}(0)([X^+_{a,b}(n_t),X^+_{b+1,l}(0)]_{q^{-1}}-X^+_{a,b}(n_t)X^+_{b+1,l}(0))v^{(t-1)}_{a,b}\\
&&-q^{-2}X^+_k(m)X^+_{b+1,l}(0)X^+_{a,b}(n_t)X^+_{k+1,b}(0)v^{(t-1)}_{a,b}\\
&&=-qX^+_k(m)X^+_{k+1,b}(0)X^+_{a,l}(n_t)v^{(t-1)}_{a,b}+qX^+_k(m)X^+_{k+1,b}(0)X^+_{a,b}(n_t)X^+_{b+1,l}(0)v^{(t-1)}_{a,b}\\
&&-q^{-2}X^+_k(m)X^+_{b+1,l}(0)X^+_{a,b}(n_t)X^+_{k+1,b}(0)v^{(t-1)}_{a,b},
\end{eqnarray*}
where the first two terms  on the right hand side vanishe by \eqref{eq:induct}, and the third can be manipulated to yield
\begin{eqnarray*}
&&-q^{-2}X^+_k(m)[X^+_{b+1,l}(0),X^+_{a,b}(n_t)]_qX^+_{k+1,b}(0)v^{(t-1)}_{a,b}+
    X^+_k(m)X^+_{a,b}(n_t)X^+_{k+1,l}(0)v^{(t-1)}_{a,b}\\
&&=q^{-1}X^+_k(m)X^+_{a,l}(0)X^+_{k+1,b}(0)v^{(t-1)}_{a,b}+
    X^+_k(m)X^+_{a,b}(n_t)X^+_{k+1,l}(0)v^{(t-1)}_{a,b}\\
&&=X^+_k(m)X^+_{a,b}(n_t)X^+_{k+1,l}(0)v^{(t-1)}_{a,b}\\
&&=[X^+_k(m), [[X^+_{a,k-1}(n_t),[X^+_{k}(0),X^+_{k+1}(0)]_{q_{k+1}}]_{q_{k}},X^+_{k+2,b}(0)]_{q_{k+2}}]X^+_{k+1,l}(0)v^{(t-1)}_{a,b}\\
&&=q^{m}[X^+_k(m), [[X^+_{a,k-1}(n_t),[X^+_{k}(m),X^+_{k+1}(-m)]_{q_{k+1}}]_{q_{k}},X^+_{k+2,b}(0)]_{q_{k+2}}]X^+_{k+1,l}(0)v^{(t-1)}_{a,b}\\
&&+\sum\limits_{s=0}^{|m|}c_s[X^+_k(m), [[X^+_{a,k-1}(n_t),X^+_{k+1}(\pm s)X^+_{k}(\mp s)]_{q_{k}},X^+_{k+2,b}(0)]_{q_{k+2}}]X^+_{k+1,l}(0)v^{(t-1)}_{a,b},
\end{eqnarray*}
where the first  term on the right hand side vanishes by \eqref{lem:relat}. Hence we can rewrite the right hand side as
\begin{eqnarray*}
&&\sum\limits_{s=0}^{|m|}c_s[X^+_k(m), [X^+_{a,k-1}(n_t),X^+_{k+1,b}(\pm s)X^+_{k}(\mp s)]_{q_k}]X^+_{k+1,l}(0)v^{(t-1)}_{a,b}\\
&&=\sum\limits_{s=0}^{|m|}c_sX^+_k(m)X^+_{a,k-1}(n_t)X^+_{k+1,b}(\pm s)X^+_{k}(\mp s)X^+_{k+1,l}(0)v^{(t-1)}_{a,b}\\
&&=\sum\limits_{s=0}^{|m|}c_sX^+_k(m)X^+_{a,k-1}(n_t)X^+_{k+1,b}(\pm s)X^+_{k,l}(\mp s)v^{(t-1)}_{a,b}\\
&&= 0.
\end{eqnarray*}

Now we consider the case   $k= a=M, l>b$. Since
\[
X^+_{M,l}(m)v^{(t)}_{M,b}=[X^+_{M,b+1}(m),X^+_{b+2,l}(0)]_{q^{-1}}v^{(t)}_{M,b}
=-q^{-1}X^+_{b+2,l}(0)X^+_{M,b+1}(m)v^{(t)}_{M,b},
\]
it is sufficient to show that $X^+_{M,b+1}(m)v^{(t)}_{M,b}=0.$

By \eqref{>M} and $X^+_i(m)v^{(t-1)}_{M,b}=X^+_i(m)v^{(t)}_{M,b}=0, i\ne M$, we have
\[
X^+_{M,b+1}(m)X^+_{M,b}(n_t)v^{(t-1)}_{M,b}=[X^+_{M,b}(0),X^+_{b+1}(m)]_{q^{-1}} [X^+_{M,b-1}(0),X^+_b(n_t)]_{q^{-1}}v^{(t-1)}_{M,b}.
\]
Hence,
\begin{eqnarray*}
\begin{aligned}
&X^+_{M,b+1}(m)v^{(t)}_{M,b}=X^+_{M,b+1}(m)X^+_{M,b}(n_t)v^{(t-1)}_{M,b}\\
&=[[X^+_{M,b}(0),X^+_{b+1}(m)]_{q^{-1}}, [X^+_{M,b-1}(0),X^+_b(n_t)]_{q^{-1}}]_qv^{(t-1)}_{M,b}\\
&=[[[X^+_{M,b}(0),X^+_{b+1}(m)]_{q^{-1}},X^+_{M,b-1}(0)]_q,X^+_b(n_t)]_{q^{-1}}v^{(t-1)}_{M,b}\\
&-q[X^+_{M,b-1}(0),[[X^+_{M,b}(0),X^+_{b+1}(m)]_{q^{-1}},X^+_b(n_t)]]_{q^{-2}}v^{(t-1)}_{M,b}\\
&=-q[X^+_{M,b-1}(0),[[X^+_{M,b-2}(0),[X^+_{b-1}(0),[X^+_b(0),X^+_{b+1}(m)]_{q^{-1}}]_{q^{-1}}]_{q^{-1}},X^+_b(n_t)]]_{q^{-2}}v^{(t-1)}_{M,b}\\
&=[X^+_{M,b-1}(0),[X^+_{M,b-2}(0),[[X^+_{b-1}(0),[X^+_{b+1}(m),X^+_{b}(0)]_{q}]_{q^{-1}},X^+_b(n_t)]]_{q^{-1}}]_{q^{-2}}v^{(t-1)}_{M,b}\\
&=q^{n_t}[X^+_{M,b-1}(0),[X^+_{M,b-2}(0),[[X^+_{b-1}(0),[X^+_{b+1}(m-n_t),X^+_{b}(n_t)]_{q}]_{q^{-1}},X^+_b(n_t)]]_{q^{-1}}]_{q^{-2}}v^{(t-1)}_{M,b}\\
&+\sum\limits_{s=0}^{|n_t|}c_s
[X^+_{M,b-1}(0),[X^+_{M,b-2}(0),[[X^+_{b-1}(0),X^+_{b}(\pm s)X^+_{b+1}(m\mp s)]_{q^{-1}},X^+_b(n_t)]]_{q^{-1}}]_{q^{-2}}v^{(t-1)}_{M,b}\\
&=\sum\limits_{s=0}^{|n_t|}c_s
[X^+_{M,b-1}(0),[[X^+_{M,b-1}(0),X^+_{b}(\pm s)X^+_{b+1}(m\mp s)]_{q^{-1}},X^+_b(n_t)]]_{q^{-2}}v^{(t-1)}_{M,b}\\
&=\sum\limits_{s=0}^{|n_t|}c_s
[X^+_{M,b-1}(0),[[X^+_{M,b-1}(0),X^+_{b}(\pm s)]_{q^{-1}}X^+_{b+1}(m\mp s),X^+_b(n_t)]]_{q^{-2}}v^{(t-1)}_{M,b}\\
&=\sum\limits_{s=0}^{|n_t|}c_s
[[X^+_{M,b-1}(0),X^+_{b}(\pm s)]_{q^{-1}}X^+_{b+1}(m\mp s),X^+_b(n_t)]X^+_{M,b-1}(0)v^{(t-1)}_{M,b}\\
&=\sum\limits_{s=0}^{|n_t|}c_s
[X^+_{M,b-1}(0),X^+_{b}(\pm s)]_{q^{-1}}X^+_{b+1}(m\mp s)X^+_b(n_t)X^+_{M,b-1}(0)v^{(t-1)}_{M,b}\\
&+X^+_b(n_t)[X^+_{M,b-1}(0),X^+_{b}(\pm s)]_{q^{-1}}X^+_{b+1}(m\mp s)X^+_{M,b-1}(0)v^{(t-1)}_{M,b}\\
&=\sum\limits_{s=0}^{|n_t|}c_s
[X^+_{M,b-1}(0),X^+_{b}(\pm s)]_{q^{-1}}X^+_{b+1}(m\mp s)X^+_b(n_t)X^+_{M,b-1}(0)v^{(t-1)}_{M,b}\\
&=\sum\limits_{s=0}^{|n_t|}c_s
[X^+_{M,b-1}(0),X^+_{b}(\pm s)]_{q^{-1}}X^+_{M,b+1}(m\mp s+n_t)v^{(t-1)}_{M,b}\\
&=0.
\end{aligned}
\end{eqnarray*}
This completes the proof of the Lemma.
\end{proof}

\section{Proof of Lemma \ref{key-2}}\label{pf:key-2}
\begin{proof}[Proof of Lemma \ref{key-2}]
Note that \eqref{3-} directly follows from \eqref{3}. Hence we only need to prove \eqref{3}.

We first show that
\begin{equation}\label{0-1-0}
X^+_{a,b}(p)X^+_{a,b}(k)X^+_{a,b}(l)v_{a,b}=0 \quad \mbox{ for all }p, k, l\in\Z\mbox{ with } p\equiv l~({\rm mod}~2).
\end{equation}

For $a<M$, we have
\begin{eqnarray*}
        &&[X^+_{a,b}(p),X^+_{a,b}(p+1)]\\
&& =[[X^+_{a}(p),X^+_{a+1,b}(0)]_q,[X^+_{a}(p+1),X^+_{a+1,b}(0)]_q]\\
&&=[[X^+_{a}(p),X^+_{a+1,b}(0)]_q,X^+_{a}(p+1)]_{q^{-1}},X^+_{a+1,b}(0)]_{q^2}\\
   &&+q^{-1}[X^+_{a}(p+1),[[X^+_{a,b}(p),X^+_{a+1,b}(0)]_{q}]_{q^2}\\
&&=[[[X^+_{a}(p),X^+_{a+1}(0)]_{q},X^+_{a}(p+1)]_{q^{-1}},X^+_{a+2,b}(0)]_{q_{a+2}},X^+_{a+1,b}(0)]\\
&&=-q[[[X^+_{a+1}(0),X^+_{a}(p)]_{q^{-1}},X^+_{a}(p+1)]_{q^{-1}},X^+_{a+2,b}(0)]_{q_{a+2}},X^+_{a+1,b}(0)]\\
&&=q[[[X^+_{a}(p+1),X^+_{a+1}(-1)]_{q^{-1}},X^+_{a}(p+1)]_{q^{-1}},X^+_{a+2,b}(0)]_{q_{a+2}},X^+_{a+1,b}(0)]\\
&&=0.
\end{eqnarray*}
Similarly, one can show that $[X^+_{M,b}(p),X^+_{M,b}(p+1)]=0, b>M.$

From \eqref{2}, we have
\[
X^+_{a,b}(p)X^+_{a,b}(p\pm 1)X^+_{a,b}(l)v_{a,b}=0\quad \mbox { for all }p\equiv l~({\rm mod}~2).
\]
This establishes \eqref{0-1-0} for $|p-k|=1$.

We now use induction on $|p-k|$ to prove
\eqref{0-1-0}. By the induction hypothesis, for all $ p, k, l\in \Z$ with $p\equiv l~({\rm mod}~2),|p-k|\le 2i-1$,
\begin{equation}\label{induction}
X^+_{a,b}(p)X^+_{a,b}(k)X^+_{a,b}(l)v_{a,b}=0.
\end{equation}
We now consider $X^+_{a,b}(p)X^+_{a,b}(p+1+2i)X^+_{a,b}(l)v_{a,b}$.

For $a<M$,  we have
\begin{eqnarray*}
&&[X^+_{a,b}(p),X^+_{a,b}(p+1+2i)]\\
&&=[X^+_{a,b}(p),[X^+_{a}(p+1+2i),X^+_{a+1,b}(0)]_q]\\
&&=[[X^+_{a,b}(p),X^+_{a}(p+1+2i)]_{q^{-1}},X^+_{a+1,b}(0)]_{q^2}\\
&&+q^{-1}[X^+_a(p+1+2i),[X^+_{a,b}(p),X^+_{a+1,b}(0)]_q]_{q^2}
\end{eqnarray*}
where the second  term on the right hand side vanishes by Lemma \ref{lem:relat}, and the first can be rewritten as
\begin{eqnarray*}
 &&[[[[X^+_{a}(p),X^+_{a+1}(0)]_q,X^+_{a}(p+1+2i)]_{q^{-1}},X^+_{a+2,b}(0)]_{q_{a+2}},X^+_{a+1,b}(0)]_{q^2}\\
&&=q^{2i}[[[[X^+_{a}(p+2i),X^+_{a+1}(-2i)]_q,X^+_{a}(p+1+2i)]_{q^{-1}},X^+_{a+2,b}(0)]_{q_{a+2}},X^+_{a+1,b}(0)]_{q^2}\\
     &&+ \sum\limits_{s=1}^{2i}c_s[[ [X^+_{a+1}(-s)X^+_{a}(p+s),X^+_{a}(p+1+2i)]_{q^{-1}},X^+_{a+2,b}   (0)]_{q_{a+2}},X^+_{a+1,b}(0)]_{q^2},
\end{eqnarray*}
We note that  the first  term on the right hand side vanishes:
\begin{eqnarray*}
&&[[X^+_{a}(p+2i),X^+_{a+1}(-2i)]_q,X^+_{a}(p+1+2i)]_{q^{-1}},X^+_{a+2,b}(0)]_{q_{a+2}}\\
&&=-q[[X^+_{a+1}(-2i),X^+_{a}(p+2i)]_{q^{-1}},X^+_{a}(p+1+2i)]_{q^{-1}},X^+_{a+2,b}(0)]_{q_{a+2}}\\
&&=q[[X^+_{a}(p+2i+1),X^+_{a+1}(-2i-1)]_{q^{-1}},X^+_{a}(p+1+2i)]_{q^{-1}},X^+_{a+2,b}(0)]_{q_{a+2}}\\
&&=0.
\end{eqnarray*}
 Hence
\begin{eqnarray*}
&& [X^+_{a,b}(p),X^+_{a,b}(p+1+2i)]\\
&&= \sum\limits_{s=1}^{2i}c_s[[ [X^+_{a+1}(-s)X^+_{a}(p+s),X^+_{a}(p+1+2i)]_{q^{-1}},X^+_{a+2,b}   (0)]_{q_{a+2}},X^+_{a+1,b}(0)]_{q^2}\\
&&=\sum\limits_{s=1}^{2i}c_s[ [X^+_{a+1,b}(-s)X^+_{a}(p+s),X^+_{a}(p+1+2i)]_{q^{-1}},X^+_{a+1,b}(0)]_{q^2}.
\end{eqnarray*}
By \eqref{2},  $X^+_{a,b}(p)X^+_{a,b}(p+1+2i)X^+_{a,b}(l)v_{a,b}
=[X^+_{a,b}(p),X^+_{a,b}(p+1+2i)]X^+_{a,b}(l)v_{a,b}$.  Hence
\begin{eqnarray*}
&&X^+_{a,b}(p)X^+_{a,b}(p+1+2i)X^+_{a,b}(l)v_{a,b}\\
&&=\sum\limits_{s=1}^{2i}c_s[ [X^+_{a+1,b}(-s)X^+_{a}(p+s),X^+_{a}(p+1+2i)]_{q^{-1}},X^+_{a+1,b}(0)]_{q^2}
X^+_{a,b}(l)v_{a,b}.
\end{eqnarray*}
We can rewrite the right hand side as
\begin{eqnarray*}
&&\sum\limits_{s=1}^{2i}c_s\left(X^+_{a+1,b}(-s)X^+_a(p+s)X^+_a(p+1+2i)\right.\\
&&\left.-q^{-1}X^+_a(p+1+2i)X^+_{a+1,b}(-s)X^+_a(p+s)\right)X^+_{a+1,b}(0)X^+_{a,b}(l)v_{a,b}\\
&&=\sum\limits_{s=1}^{2i} c_sX^+_{a+1,b}(-s)X^+_a(p+s)X^+_{a,b}(p+1+2i)X^+_{a,b}(l)v_{a,b}\\
&&-\sum\limits_{s=1}^{2i}q^{-1} c_sX^+_a(p+1+2i)X^+_{a+1,b}(-s)X^+_{a,b}(p+s)X^+_{a,b}(l)v_{a,b}.
\end{eqnarray*}
Note that the first term on the right hand side vanishes by
\eqref{2}, and by using \eqref{2} and \eqref{ii+1} we can rewrite the second term as
\begin{eqnarray*}
&&-\sum\limits_{s=1}^{2i} q^{-1}c_s[[X^+_a(p+1+2i),X^+_{a+1}(-s)]_{q},X^+_{a+2,b}(0)]_{q_{a+2}}X^+_{a,b}(p+s)X^+_{a,b}(l)v_{a,b}\\
&&=\sum\limits_{s=1}^{2i} d_s[[X^+_a(p+1+2i-s),X^+_{a+1}(0)]_q,X^+_{a+2,b}(0)]_{q_{a+2}}X^+_{a,b}(p+s)X^+_{a,b}(l)v_{a,b}\\
&&+\sum\limits_{s=1}^{2i} d'_s\sum\limits_{r=0}^{s-1} c_r[X^+_{a+1}(-s+ r)X^+_{a}(p+1+2i- r),X^+_{a+2,b}(0)]_{q_{a+2}}X^+_{a,b}(p+s)X^+_{a,b}(l)v_{a,b}.
\end{eqnarray*}
The second term on the right hand side varnishes by \eqref{2}.  This leads to
\begin{eqnarray*}
&&X^+_{a,b}(p)X^+_{a,b}(p+1+2i)X^+_{a,b}(l)v_{a,b}\\
&&=\sum\limits_{s=1}^{2i}d_sX^+_{a,b}(p+1+2i-s)X^+_{a,b}(p+s)X^+_{a,b}(l)v_{a,b}.
\end{eqnarray*}
We observe that   $|(p+1+2i-s)-(p+s)|=|1+2i-2s|\le 2i$ for $1\le s\le 2i-1$. Thus $X^+_{a,b}(p)X^+_{a,b}(p+1+2i)X^+_{a,b}(l)v_{a,b}=0$ by \eqref{induction}.

\medskip

For $a=M$, by using \eqref{MM+1b}, we  obtain
\begin{eqnarray*}
&&[X^+_{M,b}(p),X^+_{M,b}(p+1+2i)]
=[X^+_{M,b}(p),[X^+_M(p+1+2i),X^+_{M+1,b}(0)]_{q^{-1}}]\\
&&=[[X^+_{M,b}(p),X^+_M(p+1+2i)]_{q^{-1}},X^+_{M+1,b}(0)]\\
&&-q^{-1}[X^+_M(p+1+2i),[X^+_{M,b}(p),X^+_{M+1,b}(0)]_q],
\end{eqnarray*}
where the second term on the right hand side varnishes by Lemma \ref{lem:relat}.
We note that
\begin{eqnarray*}
&&[X^+_{M,b}(p),X^+_M(p+1+2i)]_{q^{-1}}\\
&&=[[[X^+_{M}(p),X^+_{M+1}(0)]_{q^{-1}},X^+_M(p+1+2i)]_{q^{-1}}, X^+_{M+2,b}(0)]_{q^{-1}}\\
&&=q^{-1}[[[X^+_{M}(p+1),X^+_{M+1}(-1)]_{q},X^+_M(p+1+2i)]_{q^{-1}}, X^+_{M+2,b}(0)]_{q^{-1}}\\
&&=q^{2i-1}[[[X^+_{M}(p+1+2i),X^+_{M+1}(-1-2i)]_{q},X^+_M(p+1+2i)]_{q^{-1}}, X^+_{M+2,b}(0)]_{q^{-1}}\\
&&+\sum\limits_{s=0}^{2i-1}c_s[[X^+_{M}(p+1+s)X^+_{M+1}(-1-s),X^+_M(p+1+2i)]_{q^{-1}}, X^+_{M+2,b}(0)]_{q^{-1}}\\
&&=\sum\limits_{s=0}^{2i-1}c_s[X^+_{M}(p+1+s)X^+_{M+1,b}(-1-s),X^+_M(p+1+2i)]_{q^{-1}}.
\end{eqnarray*}
Hence
\begin{eqnarray*}
&&[X^+_{M,b}(p),X^+_{M,b}(p+1+2i)]\\
&&=\sum\limits_{s=0}^{2i-1}c_s[[X^+_{M}(p+1+s)X^+_{M+1,b}(-1-s),X^+_M(p+1+2i)]_{q^{-1}},X^+_{M+1,b}(0)].
\end{eqnarray*}
Now we have
\begin{eqnarray*}
\begin{aligned}
&X^+_{M,b}(p)X^+_{M,b}(p+1+2i)X^+_{M,b}(l)v_{M,b}\\
&=
[X^+_{M,b}(p),X^+_{M,b}(p+1+2i)]X^+_{M,b}(l)v_{M,b}\quad \mbox{ by } \eqref{induction}\\
&=\sum\limits_{s=0}^{2i-1}c_s[[X^+_{M}(p+1+s)X^+_{M+1,b}(-1-s),X^+_M(p+1+2i)]_{q^{-1}},X^+_{M+1,b}(0)]X^+_{M,b}(l)v_{M,b}\\
&=\sum\limits_{s=0}^{2i-1}c_s X^+_{M+1,b}(0)X^+_{M}(p+1+s)X^+_{M+1,b}(-1-s)X^+_M(p+1+2i)X^+_{M,b}(l)v_{M,b}\quad \mbox{ by }\eqref{2}\\
&=\sum\limits_{s=0}^{2i-1}c_s X^+_{M+1,b}(0)X^+_{M}(p+1+s)[X^+_{M+1,b}(-1-s),X^+_M(p+1+2i)]_qX^+_{M,b}(l)v_{M,b}
\end{aligned}
\end{eqnarray*}
By \eqref{>M+} and \eqref{2}, we note that,
\begin{eqnarray*}
&&[X^+_{M+1,b}(-1-s),X^+_M(p+1+2i)]_qX^+_{M,b}(l)v_{M,b}\\
&&=[[X^+_{M+1}(-1-s),X^+_M(p+1+2i)]_q,X^+_{M+2,b}(0)]_{q^{-1}}X^+_{M,b}(l)v_{M,b}\\
&&=\sum\limits_{s=0}^{2i-1} q^{-1-s}[[X^+_{M+1}(0),X^+_M(p+2i-s)]_q,X^+_{M+2,b}(0)]_{q^{-1}}X^+_{M,b}(l)v_{M,b}\\
&&+\sum\limits_{s=0}^{2i-1} \sum\limits_{r=0}^s c_r [X^+_M(p+1+2i-r)X^+_{M+1}(r-1-s),X^+_{M+2,b}(0)]_{q^{-1}}X^+_{M,b}(l)v_{M,b}\\
&&=\sum\limits_{s=0}^{2i-1} q^{-1-s}[[X^+_{M+1}(0),X^+_M(p+2i-s)]_q,X^+_{M+2,b}(0)]_{q^{-1}}X^+_{M,b}(l)v_{M,b}\quad \mbox{ by }\eqref{2}\\
&&=\sum\limits_{s=0}^{2i-1} q^{-s}X^+_{M,b}(p+2i-s)X^+_{M,b}(l)v_{M,b}.
\end{eqnarray*}
Hence,
\begin{eqnarray*}
&&X^+_{M,b}(p)X^+_{M,b}(p+1+2i)X^+_{M,b}(l)v_{M,b}\\
&&=\sum\limits_{s=0}^{2i-1}c_s q^{-s}X^+_{M+1,b}(0)X^+_{M}(p+1+s)X^+_{M,b}(p+2i-s)X^+_{M,b}(l)v_{M,b}\\
&&=\sum\limits_{s=0}^{2i-1}c_s q^{-s}[X^+_{M+1,b}(0)X^+_{M}(p+1+s)]_qX^+_{M,b}(p+2i-s)X^+_{M,b}(l)v_{M,b}\quad \mbox{ by }\eqref{2}\\
&&=-\sum\limits_{s=0}^{2i-1}c_sq^{1-s} X^+_{M,b}(p+1+s)X^+_{M,b}(p+2i-s)X^+_{M,b}(l)v_{M,b}.
\end{eqnarray*}
We observe that $X^+_{M,b}(p+2i)X^+_{M,b}(l)v_{M,b}=0$ by \eqref{2} since $p+2i\equiv l ({\rm mod} 2)$, and $|p+2i-s-(p+1+s)|\le 2i-1$ for $1\le s\le 2i-1$. Hence, from \eqref{induction} and  the above equality  we have
$X^+_{M,b}(p)X^+_{M,b}(p+1+2i)X^+_{M,b}(l)v_{M,b}=0.$
Now we have prove that
\begin{equation}\label{1+2i}
X^+_{a,b}(p)X^+_{a,b}(p+1+2i)X^+_{a,b}(l)v_{a,b}=0\quad (a,b)\in S.
\end{equation}
Similarly, one can prove that
\begin{equation}\label{-1-2i}
X^+_{a,b}(p)X^+_{a,b}(p-1-2i)X^+_{a,b}(l)v_{a,b}=0.
\end{equation}
This completes the proof of  \eqref{0-1-0}.

Using  the similar arguments  in the proof of \eqref{0-1-0} one can prove that
\begin{equation}\label{0-0-1}
X^+_{a,b}(p)X^+_{a,b}(k)X^+_{a,b}(l)v_{a,b} =0\quad \mbox{ for all }p, k, l\in\Z\mbox{ with } p\equiv k~({\rm mod}~2).
\end{equation}
Now \eqref{3} follows from \eqref{0-1-0}, \eqref{0-0-1} and \eqref{2}.
\end{proof}

\vspace{.5cm}

\noindent {\bf Acknowledgement.}
This work was supported by the Chinese National Natural Science Foundation grant No. 11271056,
Australian Research Council Discovery-Project Grants DP0986349 and DP140103239, Qing Lan Project of Jiangsu
Province, and Jiangsu Overseas Research and Training Program for Prominent Young and Middle Aged University Teachers and Presidents. Part of this work was completed when both authors visited the University of Science and Technology of China.

\end{document}